\documentclass[11pt,a4paper]{article}

\usepackage[T1]{fontenc}
\usepackage[cp1250]{inputenc}

\usepackage{amssymb,amsmath,amsthm,amsfonts,stmaryrd} 
\usepackage{amscd}
\usepackage{bm} 
\input xy
\xyoption{all}

\usepackage{color}

\usepackage{enumerate} 
\usepackage{hyperref} 

\newtheorem{thm}{Theorem}[section]
\newtheorem{prop}[thm]{Proposition}
\newtheorem{lem}[thm]{Lemma}

\theoremstyle{definition}

\newtheorem{remark}[thm]{Remark}
\newtheorem{definition}[thm]{Definition}

\numberwithin{equation}{section}

\newcommand{\VF}{\mathfrak{X}} 
\newcommand{\X}{\mathcal{X}} 
\newcommand{\A}{\mathcal{A}} 
\newcommand{\B}{\mathcal{B}} 
\newcommand{\VV}{\mathcal{V}} 
\newcommand{\Adm}{\mathcal{ADM}} 
\newcommand{\Map}{\mathcal{MAP}}
\newcommand{\R}{\mathbb{R}}
\newcommand{\RR}{\mathcal{R}}
\newcommand{\eps}{\varepsilon}
\newcommand{\del}{\delta}
\newcommand{\g}{\mathfrak{g}}
\newcommand{\dT}{\mathrm{d_T}}

\newcommand{\ra}{\rightarrow}
\newcommand{\lra}{\longrightarrow}
\newcommand{\rel}{\rightarrowtriangle}

\newcommand{\Crit}{\operatorname{Crit}}
\newcommand{\Sec}{\operatorname{Sec}}
\newcommand{\id}{\operatorname{id}}
\newcommand{\dd}{\mathrm{d}} 
\newcommand{\sT}{\textrm{T}}
\newcommand{\T}{\mathrm{T}} 
\newcommand{\V}{\mathrm{V}} 
\newcommand{\pa}{\partial} 

\newcommand{\wh}{\widehat}
\newcommand{\wt}{\widetilde}
\newcommand{\ol}{\overline}

\def\<#1>{\big\langle #1\big\rangle} 

\begin{document}

\title{ Jacobi vector fields for  Lagrangian systems on  algebroids}
\date{\today}
\author{Micha\l \ J\'o\'zwikowski\footnote{This research was supported by the Polish Ministry of Science and Higher Education under the grant N N201 416839.}\\
 Institute of Mathematics, Polish Academy of Sciences\\\'Sniadeckich 8,
P.O. Box 21, 00-956 Warszawa, Poland\\{\tt mjozwikowski@gmail.com}
}
\maketitle
\begin{abstract}
We study the geometric nature of the Jacobi equation. In particular we prove that Jacobi vector fields (JVFs) along a solution of the Euler-Lagrange (EL) equations are themselves solutions of the EL equations but considered on a non-standard algebroid (different from the tangent bundle Lie algebroid). As a consequence we obtain a simple non-computational proof of the relation between
the null subspace of the second variation of the action and the presence of JVFs (and conjugate points) along an extremal. We work in the framework of skew-symmetric  algebroids. 
\bigskip

\noindent\emph{MSC 2010: 70H03, 53D17, 70H25, 17B66}

\bigskip

\noindent
\emph{Key words: Jacobi field, conjugate points, Lagrangian dynamics, Lie algebroids, variational principle, second variation}
\end{abstract}

\section{Introduction}\label{sec:intro}

\emph{Jacobi vector fields} (JVFs) and \emph{conjugate points} appear naturally in Riemannian geometry in connection with geodesics. There are at least three different ways to speak about JVFs. Consider a Riemannian manifold $(M,g)$ and let $\zeta:[t_0,t_1]\ra M$ be a geodesic. 

\begin{enumerate}[(A)]
\item First of all, we may define a JVF  in geometric terms: by means of the Levi-Civita connection $\nabla$ and the Riemannian curvature tensor $R(\cdot,\cdot)$ of $(M,g)$. It is a vector field $J(t)$ along $\zeta(t)$ which satisfies the following second order differential equation 
\begin{equation}\label{eqn:Jacobi_geod}
\nabla_{\dot\zeta}\nabla_{\dot\zeta}J+R(J,\dot\zeta)\dot\zeta=0.
\end{equation} 
\item In fact \eqref{eqn:Jacobi_geod} (known as the \emph{Jacobi equation}) appears naturally when considering the tangent prolongation of the geodesic equation $\nabla_{\dot\zeta}\dot\zeta=0$. Strictly speaking if $t\mapsto\zeta(t,s)$, where $s\in(-\eps,\eps)$ and $\zeta(t,0)=\zeta(t)$, is a one-parameter family of geodesics, then $J(t):=\pa_s\big|_{s=0}\zeta(t,s)$ satisfies \eqref{eqn:Jacobi}. Conversely, every solution of \eqref{eqn:Jacobi_geod} may be constructed in this way. This gives the second interpretation of JVFs. 
\item Finally, JVFs appear in connection with second variation of the energy functional $S(\zeta)=\frac 12\int_{t_0}^{t_1}g(\dot\zeta,\dot\zeta)\dd t$. Precisely, if the first variation $\del S(\zeta)$ vanishes (i.e. $\zeta$ is a geodesics), then the second variation reads as (see \cite{Milnor})
\begin{equation}\label{eqn:C}
\del^2 S(\zeta)(X,Y)=-\int_{t_0}^{t_1}g\left(Y,\nabla_{\dot\zeta}\nabla_{\dot\zeta}X+R(X,\dot\zeta)\dot\zeta\right)\dd t-g\left(Y(t),\nabla_{\dot\zeta}X(t)\right)\Big|_{t_0}^{t_1},\end{equation}
where $X$ and $Y$ are vector fields along $\zeta$ (and at least one of them is vanishing at the end-points). Inside the integral we recognize the left-hand-side of the Jacobi equation. 
\end{enumerate}
  
This third interpretation of JVFs becomes more important after introducing the notion of conjugate points. We say that $\zeta(t_0)$ and $\zeta(t_1)$ are \emph{conjugated} along $\zeta$ if there exists a non-vanishing JVF $J(t)$ along $\zeta$ which vanishes at $t_0$ and $t_1$ (it corresponds to the existence of a family of geodesics emerging from $\zeta(t_0)$ and meeting again in $\zeta(t_1)$ (up to first order terms)). The celebrated results of Morse \cite{Milnor} links the presence of conjugate points with the properties of the action $S(\cdot)$.

\begin{thm}[Morse]\label{thm:morse}
Consider $\del^2 S(\zeta)(\cdot,\cdot)$ as a symmetric bilinear form on the space of vector fields along $\zeta(t)$ which vanish at $t_0$ and $t_1$. Then,
\begin{enumerate}
\item The index (i.e. the dimension of the negative space) of $\del^2 S(\zeta)(\cdot,\cdot)$ is always finite and equals the number of points (counted with their multiplicities) conjugated to $\zeta(t_0)$, which lie on $\zeta$ between $\zeta(t_0)$ and $\zeta(t_1)$.
\item A vector field $X(t)$ along $\zeta(t)$ is a null vector of $\del^2 S(\zeta)(\cdot,\cdot)$ if and only if $X(t)$ is a JVF vanishing at $t_0$ and $t_1$ (i.e. $\zeta(t_0)$ and $\zeta(t_1)$ are conjugated along $\zeta$). The dimension of the null space equals the multiplicity of $\zeta(t_0)$ and $\zeta(t_1)$ as conjugate points. 
\end{enumerate}
\end{thm}   

In fact, much of the theory of JVFs and conjugate points can be rewritten for an arbitrary Lagrangian system $L:\T M\ra\R$. There is, of course, sense in speaking about a tangent prolongation of the Euler-Lagrange (EL) equation, i.e. defining a JVF as the first order variation associated with a one-parameter family of solutions as in (B). We may also link such JVFs with the properties of the second variations of the action functional $S(\zeta)=\int_{t_0}^{t_1}L(\dot\zeta)\dd t$ as in (C). Also the notion of conjugate points can be straightforwardly generalized to this situation but, in general, we cannot expect any results of the same type as in the first part of Theorem \ref{thm:morse} unless some additional assumptions are put on $L$. However, the second part of this theorem, i.e. the link between existence of conjugate points and the null subspace of the second variation $\del^2S$, can be maintained.   

What seems to be unclear is the geometric nature of these generalized JVFs (the description corresponding to (A)). In this paper we aim to give an answer to this problem. In fact we will consider Lagrangian systems of a more general class, namely Lagrangian systems on skew-symmetric algebroids. 

Algebroids have gained much attention in recent years providing a useful geometric language to speak about mechanics (see \cite{GG} and the references therein). However, some of the studies of these objects seem to be just vain straightforward generalizations of the classical concepts. We believe that this is not the case of this paper. Our main result states that the Jacobi equation associated with a Lagrangian system $L:E\ra\R$ on a skew-symmetric algebroid $\tau:E\ra\R$ is strictly connected with the EL equations for the lifted Lagrangian system $\dT L:\T E\ra\R$ on the tangent lift algebroid $\T\tau:\T E\ra\T M$. Note that even starting form the classical case of a Lagrangian system on the tangent bundle Lie algebroid $\tau_M:\T M\ra M$ we end up with a Lagrangian system on the algebroid $\T\tau_M:\T\T M\ra\T M$ which is not the standard tangent bundle Lie algebroid (in particular, we can interpret \eqref{eqn:Jacobi_geod} as a projection of the EL equation from $\T\tau_M:\T\T M\ra\T M$ to $\tau_M:\T M\ra M$. This example shows that skew-symmetric algebroids are the proper class of objects which allows to speak about JVFs even if we wish to restrict ourselves to system on tangent bundle Lie algebroids only.

To formulate our main results we need to introduce the notion of an admissible curve $\gamma:[t_0,t_1]\ra E$ on an algebroid $\tau:E\ra M$ and the notion of an infinitesimal admissible variation $\del_\xi\gamma\in\T_\gamma E$ of $\gamma$ generated by $\xi\in E_{\tau\circ\gamma}$. Precise definitions can be found in Section \ref{sec:geom}. At this point an intuitive interpretation will suffice. Treating $E$ as a bundle of generalized velocities (configurations space) of a physical system we may think of $\gamma$ as a true physical trajectory of this system. The variation $\del_\xi\gamma$, in turn, can be understood as a variation at the level of generalized velocities associated with a virtual displacement in the direction of the generator $\xi$. 

\begin{thm}\label{thm:main} Consider a Lagrangian system $L:E\ra\R$ on a skew-symmetric algebroid $\tau:E\ra M$ and an admissible curve $\gamma$ being a solution of the EL equation for this system. A curve $\eta\in E_{\tau\circ\gamma}$ is a JVF along $\gamma$ if and only if the associated infinitesimal admissible variation $\del_\eta\gamma\in\T_\gamma E$ is a solution of the EL equation for the lifted Lagrangian system $\dT L:\T E\ra\R$ on the tangent lift algebroid $\T\tau:\T E\ra\T M$. 
\end{thm}

What is more, from Theorem \ref{thm:main} and the standard relation of EL equations with variational principles, we almost tautologically obtain the relation between the second variation of the action functional and the existence of JVFs and conjugate points. This time we have to restrict ourselves to a certain subclass of skew-symmetric algebroids, namely \emph{almost-Lie} (\emph{AL}) \emph{algebroids}. 

\begin{thm}\label{thm:main1}
Let $L:E\ra\R$ be a Lagrangian system on an AL algebroid $\tau:E\ra M$. Consider the action 
$$\gamma:[t_0,t_1]\ra E\quad\longmapsto\quad S_L(\gamma):=\int_{t_0}^{t_1}L(\gamma(t))\dd t$$
defined on the space of admissible curves. 
If $\gamma$ is a solution of the EL equation then the second variation $\del^2S_L(\gamma)(\cdot,\cdot)$ is a well-defined bilinear form on the space of curves $(\eta(t),\xi(t))\in E_{\tau\circ\gamma(t)}\times E^0_{\tau\circ\gamma(t)}$. Moreover, $\eta(t) \in E_{\tau\circ\gamma}$ is a JVF along $\gamma$ if and only if 
$$\del^2 S_L(\gamma)(\eta,\xi)=0\quad\text{for any $\xi\in E^0_{\tau\circ\gamma}$}.$$

Moreover, if $E$ is a Lie algebroid, then 
$$\del^2S_L(\gamma)(\cdot,\cdot):E^0_{\tau\circ\gamma}\times E^0_{\tau\circ\gamma}\lra\R$$
is a symmetric bilinear form. The dimension of the null-space of this form equals the multiplicity of points $\gamma(t_0)$ and $\gamma(t_1)$ as conjugated points along $\gamma$.
\end{thm}

To sum up, our main purposes in this publication are:
\begin{itemize}
\item Defining the Jacobi equation (and its solutions, JVFs) as the tangent prolongation of the EL equation on a skew-symmetric algebroid. 
\item Giving a geometric interpretation of the Jacobi equation in terms of the EL equation on the tangent lift of the initial algebroid.
\item Studying the relation between second order variations of the action functional and the existence of JVFs and conjugate points.
\item Doing all the above in purely geometric terms avoiding working in local coordinates whenever possible. 
\end{itemize}
We are aware that some readers may prefer coordinate rather than geometric description. For them we summarize the most important formulas in Section \ref{sec:coor}.

\section{Geometric preliminaries}\label{sec:geom}

In this section we introduce some basic concepts associated with the notion of a skew-symmetric algebroid and with the notion of the tangent lift of a tensor field on a manifold. At the end we show how these two constructions join in the concept of the tangent lift of an algebroid. 

\paragraph{Notation.} We will use standard notions and conventions of differential geometry. Throughout the paper we will be working with a vector bundle $\tau:E\ra M$, its dual $\pi:E^\ast\ra M$ and related objects, such as $\T E$, $\T E^\ast$, $\T\T E$, etc. We will use  $\V E=\ker\{\T\tau:\T E\ra\T M\}$ to denote the subbundle of vectors in $\T E$ vertical w.r.t. $\tau:E\ra M$. By $\xi\in E_x$ or $\xi(t)\in E_{x(t)}$ we will denote a section of $E$ along a base path $x(t)\in M$ (such curves will always be defined on the interval $[t_0,t_1]$). Without some abuse of notation we will denote by $E_x$ the set of all such sections (along a fixed base path $x(t)$) and by $E^0_x$ all such sections vanishing at the end-points (i.e. $\xi(t_0)$ and $\xi(t_1)$ are null vectors w.r.t. the vector bundle structure $\tau:E\ra M$). If a manifold carries more than one vector bundle structure (like in the case of $\T E$) we will emphasize which vector bundle structure we are speaking about. For a tangent bundle $\T M$ we will write rather $\T_xM$ than $\left(\T M\right)_x$. We do not introduce any local coordinates at this stage. Readers interested in the coordinate description of the theory should read the text parallel with Section \ref{sec:coor} where such a description is given.

\subsection{Algebroids}

\begin{definition}\label{def:algebroid}
A \emph{skew-symmetric algebroid} (or shortly a \emph{skew algebroid}) structure on a vector bundle $\tau:E\ra M$ is a skew-symmetric bilinear bracket $[\cdot,\cdot]$ on the space $\Sec(\tau)$ of sections of $\tau$, together with a vector bundle morphism $\rho:E\ra \T M$ (the \emph{anchor map}) which satisfy the Leibniz rule: 
\begin{align}
\label{eqn:lieb_rule}
&[X,f\cdot Y]=f[X,Y]+\rho(X)(f)Y 
\intertext{for every $X,Y\in\Sec(\tau)$ and $f\in C^\infty(M)$. If, in addition,  $\rho$ maps the algebroid bracket into the standard Lie bracket on $\T M$, i.e.}
\label{eqn:ala}
&\rho\left([X,Y]\right)=[\rho(X),\rho(Y)]_{\sT M},
\intertext{we call $E$ an \emph{almost-Lie} (\emph{AL}) \emph{algebroid}. If $E$ is AL and, moreover,}
\label{eqn:lie}
&\left(\Sec(\tau),[\cdot,\cdot]\right)\quad \text{is a Lie algebra,}
\end{align}
we speak about \emph{Lie algebroids}.
\end{definition}

\paragraph{Algebroids as DVB morphisms.} There are several equivalent definitions of the algebroid structure on $E$. For example, it can be understood as a linear bivector field $\Pi$ on the dual bundle $\pi:E^\ast\ra M$. Such a $\Pi$ defines a double vector bundle (DVB) morphism $\wt\Pi:\T^\ast E^\ast\ra\T E^\ast$ given by contraction. More information on $\Pi$, how it is constructed, its coordinate description and how the additional conditions (e.g. $E$ being Lie or AL, etc.) are reflected in it can be found in \cite{GG, GGU}. Let us only remark that $\Pi$ is a linear Poisson tensor if and only if $E$ is a Lie algebroid.

Composition of $\wt\Pi$ with the canonical DVB-isomorphism $\RR_E:\T^\ast E\ra\T^\ast E^\ast$ (for more information on $\RR_E$ consult the Appendix)  gives a DVB-morphism $\eps:=\wt\Pi\circ\RR_E:\T^\ast E\ra\T E^\ast$ over the identity on $E^\ast$:
\begin{equation}\label{diag:eps}
\xymatrix{
 & \sT^\ast E \ar[rrr]^{\epsilon} \ar[dr]^{\pi_E}
 \ar[ddl]_{\sT^\ast\tau}
 & & & \sT E^\ast\ar[dr]^{\sT\pi}\ar[ddl]_/-20pt/{\tau_{E^\ast}}
 & \\
 & & E\ar[rrr]^/-20pt/{\rho}\ar[ddl]_/-20pt/{\tau}
 & & & \sT M \ar[ddl]_{\tau_M}\\
 E^\ast\ar[rrr]^/-20pt/{id}\ar[dr]^{\pi}
 & & & E^\ast\ar[dr]^{\pi} & &  \\
 & M\ar[rrr]^{id}& & & M &
}.
\end{equation}

This morphism gives another characterization of the algebroid structure on $E$. It will be particularly useful in the context of Lagrangian mechanics.

\paragraph{Relation $\bm\kappa$.}
Since the dual bundles of $\pi_E:\sT^\ast E\ra E$ and $\sT\pi:\sT E^\ast\ra\sT M$ are, respectively,
$\tau_E:\sT E\ra E$ and $\sT\tau:\sT E\ra\sT M$, the dual to the DVB morphism $\eps$ is a relation $\kappa:\sT E\rel\sT E$.
It is a uniquely defined smooth submanifold $\kappa\subset\T E\times\T E$ consisting of pairs $(X,Y)$ such that
$\rho(\tau_E(Y))=\T\tau(X)$ and
$$\<X,\eps(A)>_{\T\tau}=\<Y,A>_{\tau_E}$$
for any $A\in\T^\ast_{\tau_E(Y)}E$, where $\< \cdot,\cdot>_{\T\tau}$ and $\< \cdot,\cdot>_{\tau_E}$ are natural pairings. This relation can be put into the  following diagram of "double vector bundle relations":
\begin{equation}
\xymatrix{
 & \sT E  \ar[dr]^{\tau_E}
 \ar[ddl]_{\sT\tau}
 & & & \sT E\ar @{->}[lll]_{\kappa}\ar[dr]^{\sT\tau}\ar[ddl]_/-20pt/{\tau_{E}}
 & \\
 & & E\ar[rrr]^/-20pt/{\rho}\ar[ddl]_/-20pt/{\tau}
 & & & \sT M \ar[ddl]_{\tau_M}\\
 \sT M\ar[dr]^{\tau_M}
 & & & E\ar[dr]^{\tau}\ar[lll]_{\rho} & &  \\
 & M\ar[rrr]^{id}& & & M &
}.\label{eqn:kappa}
\end{equation}

\paragraph{Canonical DVB morphisms $\bm{\kappa_Q}$ and $\bm{\alpha_Q}$}
Of particular importance is the relation $\kappa$ associated with the canonical tangent bundle Lie algebroid structure on $\T Q$. We will denote it by $\kappa_Q$ to emphasize where form it originates. In fact, $\kappa_Q$ is not just a relation but a DVB-isomorphism which intertwines the two canonical projections $\tau_{\T Q}:\T\T Q\ra\T Q$ and $\T\tau_Q:\T\T Q\ra\T Q$:
$$\xymatrix{
 & \T\T Q  \ar[dr]^{\tau_{\T Q}}
 \ar[ddl]_{\T\tau_Q}
 & & & \T\T Q\ar @{->}[lll]_{\kappa_Q}\ar[dr]^{\T\tau_Q}\ar[ddl]_/-20pt/{\tau_{\T Q}}& \\
 & & \T Q\ar[rrr]^/-20pt/{\id}\ar[ddl]_/-20pt/{\tau_Q}
 & & & \T Q \ar[ddl]_{\tau_Q}\\
 \T Q\ar[dr]^{\tau_Q}
 & & & \T Q\ar[dr]^{\tau_Q}\ar[lll]_{\id} & &  \\
 & Q\ar[rrr]^{\id}& & & Q &}.
$$

A nice geometric interpretation of $\kappa_Q$ is the following. A vector $\A\in\T_Y\T Q$ can be represented by a curve $s_1\mapsto Y(s_1)\in \T_{q(t)}Q$ with $Y(0)=Y$. Since every vector itself $Y(s_1)$ can be represented by a curve $s_2\mapsto q(s_1,s_2)\in Q$ where $q(s_1,0)=q(s_1)$, the vector $\A$ can be regarded as an infinitesimal object associated with a homotopy $(s_1,s_2)\mapsto q(s_1,s_2)\in Q$. The projection  $Y=\tau_{\T Q}(\A)\in \T Q$ is an infinitesimal object associated with the curve $s_2\mapsto q(0,s_2)$, while the projection $X:=\T\tau_Q(\A)\in\T Q$ with the curve $s_1\mapsto q(s_1,0)$.  
Now we can reverse the order of variables $s_1$ and $s_2$ to obtain a new homotopy $\wt q(s_1,s_2)=q(s_2,s_1)$. The vector $\B$ corresponding to $\wt q$ is the image of $\A$ under $\kappa_Q$. Obviously, $\tau_{\T Q}(\A)=\T\tau_Q(\B)$ and $\T\tau_Q(\A)=\tau_{\T Q}(\B)$. Moreover, we see that $\kappa_Q$ is an involution (hence isomorphism).

Relation $\kappa_Q$ is dual to a DVB morphism $\eps_Q:\T^\ast\T Q\ra\T\T^\ast Q$. Since $\kappa_Q$ is an isomorphism, so is $\eps_Q$. Its inverse is usually denoted by $\alpha_Q:\T\T^\ast Q\ra\T^\ast\T Q$:

$$\xymatrix{
 & \T^\ast\T Q  \ar[dr]^{\pi_{\T Q}}
 \ar[ddl]_{\T^\ast\tau}
 & & & \T\T^\ast Q\ar @{->}[lll]_{\alpha_Q}\ar[dr]^{\T\pi_Q}\ar[ddl]_/-20pt/{\tau_{\T^\ast Q}}& \\
 & & \T Q\ar[rrr]^/-20pt/{\id}\ar[ddl]_/-20pt/{\tau_Q}
 & & & \T Q \ar[ddl]_{\tau_Q}\\
 \T^\ast Q\ar[dr]^{\pi_Q}\ar[rrr]^{\id}
 & & & \T^\ast Q\ar[dr]^{\pi_Q} & &  \\
 & Q\ar[rrr]^{\id}& & & Q &}.
$$
It is defined via the equality 
\begin{equation}\label{eqn:kappa_alpha}
\<X,\Psi>_{\T\tau_Q}=\<\kappa_Q(X),\alpha_Q(\Psi)>_{\tau_{\T Q}},
\end{equation}
for $X\in \T\T Q$ and $\Psi\in \T\T^\ast Q$ such that the first parring makes sense.

\paragraph{Admissible curves and infinitesimal admissible variations.} If $\tau:E\ra M$ carries a skew-algebroid structure, then thanks to $\kappa$ we may introduce the notion of infinitesimal admissible variations of curves in $E$. We begin with the definition of an admissible curve.
\begin{definition} A curve $\gamma:[t_0,t_1]\ra E$ is called \emph{admissible} if the following condition is satisfied
\begin{equation}\label{eqn:adm}\rho(\gamma(t))=\frac{\dd}{\dd t}\left(\tau\circ\gamma(t)\right),\quad\text{for every $t\in [t_0,t_1]$},
\end{equation}
i.e. the tangent prolongation of the base projection of $\gamma$ agrees with the anchor of $\gamma$.
\end{definition}
\noindent In the standard case $E=\T M$ admissibility of $\gamma$ means that $\gamma(t)=\dot\zeta(t)$, where $\zeta:=\tau\circ\gamma$, that is $\gamma$ is a tangent prolongation of its base projection  $\zeta$. 

Let now $\gamma:[t_0,t_1]\ra E$ be an admissible curve. A curve $\del\gamma:[t_0,t_1]\ra \T E$ which projects to $\gamma$ under $\tau_E:\T E\ra E$ is called an \emph{infinitesimal variation of $\gamma$}. Among all variations of a given $\gamma$ we distinguish the subclass of \emph{infinitesimal admissible variations}. These are constructed in the following way.

Take any path $\xi:[t_0,t_1]\ra E$ which projects to the same base path as $\gamma$, i.e. $\tau\circ \xi=\tau \circ\gamma$. The tangent prolongation $\dot\xi(t)$ is a vector field in $\T E$ along $\xi$. Now there exist an unique vector field $\del_\xi\gamma$ in $\T E$ along $\gamma$ which is $\kappa$-related to $\dot\xi(t)$, that is
$$\del_\xi\gamma(t)=\kappa_\gamma\left(\dot\xi (t)\right).$$
The lower index $\kappa_\gamma$ indicates that $\del_\xi\gamma\in\T_\gamma E$.  Note that $\xi\mapsto\del_\xi\gamma$ is linear and that from \eqref{eqn:kappa} it follows that $\del_\xi\gamma$ projects to $\rho(\xi)$ under $\T\tau$:
$$\T\tau(\del_\xi\gamma)=\rho(\xi).$$ 

We can summarize the above construction in the following definition: 
\begin{definition}\label{def:variation}
Infinitesimal variations of the form $\del_\xi\gamma$ are called \emph{admissible}. The curve $\xi$ is called an \emph{infinitesimal generator} of $\del_\xi\gamma$. We distinguish admissible \emph{infinitesimal variations with vanishing end-points} generated by $\xi$'s such that $\xi(t_i)\in E_{\tau\circ\gamma(t_i)}$ for $i=1,2$ is a null vector in $\tau:E\ra M$.
\end{definition}

The intuitive interpretation of admissible curves and admissible variations is the following. We may treat $E$ as a bundle of generalized velocities over $M$ with true velocities on $M$ obtained after applying the anchor map $\rho:E\ra\T M$. Admissible curves are these curves in the generalized configuration space $E$ which correspond to true movements. A generator  $\xi:[t_0,t_1]\ra E$ of an admissible variation can be now understood as a generalized direction in which a variation is performed and the full infinitesimal variation $\del_\zeta \gamma$ is a variation in the generalized configuration space corresponding to the movement in the direction of a generator.

\begin{remark}\label{rem:kappa_var} It may be useful for the reader to describe in detail the connection between relation $\kappa$ and infinitesimal variations in the particular case of $\kappa=\kappa_Q$. Let namely $t\mapsto q(t,s)$ be a $s$-parameterized family of curves in $Q$. It generates a variation $Y(t)=\pa_s\big|_0q(t,s)\in\T_{q(t)}Q$ along $q(t):=q(t,0)$. Denote its $t$-derivative by $\A(t):=\pa_tY(t)\in\T_{Y(t)}\T Q$. On the other hand, we can lift $q(t,s)$ to curves $X(t,s)=\pa_tq(t,s)\in T_{q(t,s)}Q$. Again $t\mapsto X(t,s)$ is an $s$-parameterized family of velocities. It defines a variation $\B(t)=\pa_s\big|_0X(t,s)\in\T_{X(t)}\T Q$ of a velocity curve $X(t):=X(t,0)=\pa_tq(t)\in\T_{q(t)}Q$. Now $\B(t)\neq\A(t)$, as the first projects to $X(t)$ and the other to $Y(t)$ under $\tau_{\T Q}$. We have, however, $\kappa_Q(\A(t))=\B(t)$. 

In other words, $\kappa_Q(\pa_tY(t))$ is a variation at the level of velocities associated with a variation $Y(t)$ at the level of base trajectories.
\end{remark}

As we have seen above, on the tangent bundle Lie algebroid infinitesimal admissible variations are infinitesimal objects associated with true variations, i.e. one-parameter families of admissible curves. It turns out that this correspondence holds if $E$ is an almost Lie algebroid (in particular, it holds for every Lie algebroid). It was first observed in \cite{GG} (see also \cite{GJ_PMP} for a detailed discussion). 
\begin{lem}\label{lem:adm_var}
Let $\tau:E\ra M$ be an AL algebroid. Then the tangent space of a Banach manifold $\Adm([t_0,t_1],E)$ of admissible paths $\gamma:[t_0,t_1]\ra E$ equals
$$\T_\gamma\Adm([t_0,t_1],E)=\{\del_\xi\gamma:\xi(t)\in E_{\tau\circ\gamma(t)}\}.$$ 
\end{lem}
Let us note that this nice interpretation is not true for general (not AL) skew-algebroids. In fact almost-Lie algebroids are precisely these skew-symmetric algebroids which posses the above property (see the references above).

\subsection{Tangent lifts}
Now we will introduce the notion of a tangent lift of a geometrical object on a manifold and apply it to a skew-algebroid. 

\paragraph{Tangent lifts of multivector fields}
A tangent lift $\dd_\T$ of a multivector field on  a manifold $Q$ was introduced in \cite{GU1} as a natural extension of  a derivation $\dd_{\T}$ of differential forms (in fact, it can be extended to the whole tensor algebra $\mathcal{T}^\bullet(Q)$).  Roughly speaking, $\dd_\T$ is a differential operator which maps tensor fields of type $\alpha$ ($\alpha$ is a multi-index) on $Q$ to tensor fields of the same type on $\T Q$. 
$$\dd_\T:\mathcal{T}^{(\alpha)}(Q)\ra\mathcal{T}^{(\alpha)}(\T Q).$$

For our purposes it will be enough to define $\dT$ on multivector fields. We will state only these results which are important to our considerations, referring to \cite{GU1} as an extensive source of information. 

Introduce local coordinates $(q^\alpha)$ on $Q$ and induced coordinates $(q^\alpha,\dot q^\beta)$ on $\T Q$. For a multivector field $X\in\VF^r(Q)$ where $X(q)=X^{\alpha_1\hdots\alpha_r}(q)\pa_{q^{\alpha_1}}\wedge\hdots\wedge\pa_{q^{\alpha_r}}$, its lift $\dT X$ is locally defined by
\begin{equation}\label{eqn:lift_coor}
\dT X(q,\dot q)=\frac{\pa X^{\alpha_1\hdots\alpha_r}(q)}{\pa q^\beta}\dot q^\beta\pa_{\dot q^{\alpha_1}}\wedge\hdots\wedge\pa_{\dot q^{\alpha_r}}+\sum_{m=1}^rX^{\alpha_1\hdots\alpha_r}(q)\pa_{\dot q^{\alpha_1}}\wedge\hdots\wedge\pa_{q^{\alpha_m}}\wedge\hdots\wedge\pa_{\dot q^{\alpha_r}}
\end{equation}
The general geometric definition can be found in \cite{GU1}. 
For us it will be sufficient to give it just in cases of fields of order $r=0$, 1 and 2. 

Let $f\in C^\infty(Q)$ be a \underline{smooth function} (a 0-vector field). We have simply 
\begin{equation}\label{eqn:lift_function}
\dT f(X)=\<\dd f(q),X>,\quad \text{where $X\in\T_qQ$};
\end{equation}
that is, $\dT f$ is a linear function on $\T Q$ defined by the differential $\dd f$.

A \underline{vector field} $V\in\VF(Q)$ can be understood as a map $\wh V:Q\ra\T Q$ covering $\id_Q$. The map $\wh{\dd_\T V}$ corresponding to the lifted vector field $\dd_\T V$ is a composition $\T \wh V\circ\kappa_Q:\T Q\ra\T\T Q$
\begin{equation}\label{eqn:lift_vector}
\xymatrix{
\T Q\ar[rr]^{\T\wh V}\ar[d]_{\id} & & \T \T Q\ar[rr]^{\kappa_Q}\ar[d]_{\T\tau_Q} &&\T\T Q\ar[d]_{\tau_{\T Q}}\\
\T Q \ar[rr]^{\id} & & \T Q\ar[rr]^\id & &\T Q
} .
\end{equation}

Another way of thinking about $\dd_\T V$ is the following. Let $\Phi^t:Q\ra Q$ be the (local) flow of $V$. The map $\T \Phi^t:\T Q\ra\T Q$ is also a (local) flow, hence it corresponds to some vector field on $\T Q$. This field is precisely $\dd_\T V$. 

Finally, a \underline{bi-vector field} $\Pi\in \VF^2(Q)$ can be associated with a vector bundle morphism $\wt\Pi:\T^\ast Q\ra\T Q$ given by contraction. The map $\wt{\dT \Pi}:\T^\ast \T Q\ra\T\T Q$ corresponding to $\dT \Pi\in\VF^2(\T Q)$  is defined by means of the following diagram:
\begin{equation}\label{eqn:lift_bivector}
\xymatrix{
\T^\ast\T Q \ar[rr]^{\wt{\dd_\T\Pi}}&& \T\T Q\\
\T\T^\ast Q \ar[u]^{\alpha_Q}\ar[rr]^{\T\wt\Pi}&& \T\T Q\ar[u]^{\kappa_Q}
}.
\end{equation}

\subsection{Tangent lift of an algebroid}

As was mentioned before, a skew-algebroid structure on $\tau:E\ra M$ is a linear bivector field $\Pi\in\VF^2(E^\ast)$. Now the tangent lift $\dT\Pi$ is a bivector field on $\T E^\ast$ linear w.r.t. $\T\pi:\T E^\ast\ra\T M$. Consequently, it defines a skew-algebroid structure on $\T\tau:\T E\ra\T M$. We will call it a \emph{tangent lift of an algebroid on $E$} and will denote it by $\dT E$. The reader can consult \cite{GU} for more information on $\dT E$. 

As has been shown in \cite{GU1}, $\dT$ commutes with the Schouten bracket, hence if $\Pi$ was Poisson then so is $\dT\Pi$. Consequently, a tangent lift of a Lie algebroid is again a Lie algebroid. One can show that the lift of an AL algebroid is an AL algebroid itself. 

In particular, $\T\T M$ posses two different Lie algebroid structures: the standard tangent bundle structure on $\tau_{\T M}:\T\T M\ra \T M$ and the lifted structure on $\T\tau_{M}:\T\T M\ra \T M$. The two are isomorphic via the map $\kappa_M$.

\paragraph{The structure of $\dT E$.} We shall now describe the structure (the morphism $\eps$, the anchor map, the relation $\kappa$) of the tangent lift Lie algebroid $\dd_\T E$. We will denote these objects by $\dT\eps$, $\dT\rho$ and $\dT\kappa$, respectively.

\begin{thm}[\cite{GU}]\label{thm:alg_lift}  For the tangent lift algebroid $\dd_\T E$ the morphism $\dd_\T \eps:\T^\ast \T E\ra \T\T E^\ast$  is defined by means of the following diagram
\begin{align}
&\xymatrix{\T^\ast\T E\ar[rr]^{\dd_\T\eps} &&\T\T E^\ast\\
\T\T^\ast E\ar[rr]^{\T\eps}\ar[u]^{\alpha_E}&& \T\T E^\ast\ar[u]^{\kappa_{E^\ast}}
}.\label{eqn:lift_eps}
\intertext{The anchor map  $\dd_\T \rho:\T E\ra\T\T M$ is the composition $\dd_T\rho=\kappa_M\circ\T\rho$ and the relation $\dd_\T \kappa:\T\T E\ra\T\T E$ is given by}
&\xymatrix{\T\T E && \T\T E\ar[ll]_{\dd_\T\kappa}\\
\T\T E \ar[u]^{\kappa_E}&& \T\T E\ar[ll]_{\T\kappa}\ar[u]^{\kappa_E}
}.\label{eqn:lift_kappa}
\end{align} 
\end{thm}

\begin{proof}
To prove \eqref{eqn:lift_eps} observe that, by definition, $\dd_\T \eps$ is the composition of $\RR_{\T E}:\T^\ast\T E\ra\T^\ast\T E^\ast$ with $\wt{\dd_T\Lambda}:\T^\ast\T E^\ast\ra\T\T E^\ast$. Composing two commutative diagrams we get the following:
$$\xymatrix{
\T^\ast \T E\ar[rr]^{\RR_{\T E}} && \T^\ast\T E^\ast\ar[rr]^{\wt{\dd_\T\Lambda}}&&\T\T E^\ast\\
\T\T^\ast E\ar[rr]^{\T\RR_E}\ar[u]_{\alpha_E}&&\T\T^\ast E^\ast \ar[rr]^{\T\wt\Lambda}\ar[u]_{\alpha_{E^\ast}}&&\T\T E^\ast\ar[u]_{\kappa_E^\ast}
}.$$ 
Since the bottom row of the above diagram is $\T\wt\Lambda\circ\T\RR_E=\T(\wt\Lambda\circ\RR_E)=\T\eps$, we have \eqref{eqn:lift_eps}. The formula for $\dd_\T\rho$ is easily obtained by projecting $\dd_\T \eps$ to the base:
$$
\xymatrix{\T\T^\ast E\ar[rr]_{\alpha_E}\ar[d]^{\T^\ast\pi_E}\ar@/^10pt/[rrrrrr]^{\T\eps} &&\T^\ast\T E\ar[rr]_{\dd_\T\eps}\ar[d]^{\pi_{\T E}} &&\T\T E^\ast\ar[d]^{\T\T\pi} &&\ar[ll]^{\kappa_E} \T\T E^\ast \ar[d]^{\T\T\pi}\\
\T E\ar[rr]^{\id} \ar@/_10pt/[rrrrrr]_{\T\rho} &&\T E\ar[rr]^{\dT\rho}&& \T\T M&&\T\T M \ar[ll]_{\kappa_M}
}.$$

Finally, \eqref{eqn:lift_kappa} can be proved as follows. By definition, $\dd_\T\kappa$ is the relation dual to the DVB morphism $\dd_\T\eps$. On the other hand, one can show that the pair of isomorphisms $(\kappa_E,\kappa_E^\ast)$ preserves the parring $\<\cdot,\cdot>_{\T\T \tau}:\T\T E\times_{\T\T M}\T\T E^\ast\ra\R$ and that, by definition, the pair of isomorphisms $(\kappa_E^{-1},\alpha_E^{-1})$ maps the parring $\<\cdot,\cdot>_{\tau_{\T E}}:\T\T E\times_{\T E}\T^\ast \T E\ra\R$ to the parring $\<\cdot,\cdot>_{\T\tau_E}:\T\T E\times_{\T E}\T\T^\ast E\ra\R$ (equation \eqref{eqn:kappa_alpha}). It follows that relation $\kappa_E^{-1}\circ\dd_\T\kappa\circ\kappa_E$ is dual to the morphism $\kappa_{E^\ast}^{-1}\circ\dd_\T\eps\circ\alpha_E\overset{\eqref{eqn:lift_eps}}{=}\T\eps$. On the other hand relation $\T\kappa$ is dual to $\T\eps$, hence $\T\kappa=\kappa_E^{-1}\circ\dd_\T\kappa\circ\kappa_E$, which is precisely \eqref{eqn:lift_kappa}. 
\end{proof}

\section{EL equations and JVFs}\label{sec:EL}

Now we will recall (after \cite{GG,GGU}) the geometric construction of the EL equations for a Lagrangian system $L:E\ra\R$ on an algebroid $\tau:E\ra M$. Later we will define JVFs for such a system as the solutions of the tangent lift of these equations. Using the machinery introduced in Section \ref{sec:geom}, we will be able to interpret JVFs as solutions of the lifted Lagrangian system $\dT L:\T E\ra\R$ on the tangent lift algebroid $\T\tau:\T E\ra\T M$.

\paragraph{Lagrangian dynamics}
By a \emph{Lagrangian system} on an algebroid $\tau:E\ra M$ we will understand a smooth function $L:E\ra\R$. Given such $L$ we can construct two maps $\lambda_L=\tau_{E^\ast}\circ\eps\circ\dd L=\T^\ast\tau\circ\dd L:E\ra E^\ast$ (the \emph{Legendre mapping}) and $\Lambda_L=\eps\circ\dd L:E\ra \T E^\ast$ (the \emph{Tulczyjew differential}) as shown on the diagram
 \begin{equation}\label{diag:EL}\xymatrix{
\T^\ast E  \ar[d]^{\pi_E}\ar[rr]^{\eps}&& \sT E^\ast \ar[d]^{\tau_{E^\ast}} \\
E\ar @{-->}[rr]^{\lambda_L}\ar @/^1pc/[u]^{\dd L} \ar@{-->}[urr]^{\Lambda_L} && E^\ast }.
\end{equation}
We define the \emph{Euler-Lagrange} (\emph{EL}) \emph{equation} for a curve $\gamma:[t_0,t_1]\ra E$ as follows:
\begin{equation}\label{eqn:EL}
\frac{\dd}{\dd t}(\lambda_L(\gamma(t)))=\Lambda_L(\gamma(t)).
\end{equation}
In other words $\gamma(t)$ is a solution of the EL equation if   $\Lambda_L(\gamma(t))$ is the tangent prolongation of $\lambda_L(\gamma(t))$. 

Observe that every such $\gamma$ is automatically admissible. Indeed,
we need a simple diagram chasing on \eqref{diag:eps} to see that $\T\pi\circ\Lambda_L(\gamma)=\rho(\gamma)$ and $\pi\circ\lambda_L(\gamma)=\tau\circ\gamma$, hence on the base level we get
$$\frac{\dd}{\dd t}\left(\tau\circ\gamma(t)\right)=\frac{\dd}{\dd t}\left(\pi\circ \lambda_L(\gamma(t))\right)=\T\pi\circ\left(\frac{\dd}{\dd t} \lambda_L(\gamma(t))\right)\overset{\eqref{eqn:EL}}=\T\pi\circ\Lambda_L(\gamma(t))=\rho(\gamma(t)).$$
 
Note also that for any admissible curve $\gamma$ (not necessarily a solution of \eqref{eqn:EL}) the difference $\Lambda_L(\gamma(t))-\frac{\dd}{\dd t}(\lambda_L(\gamma(t)))$ is a vertical vector. Indeed, 
$$\T\pi\circ\left(\Lambda_L(\gamma(t))-\frac{\dd}{\dd t} \lambda_L(\gamma(t))\right)=\T\pi\circ\Lambda_L(\gamma(t))-\frac{\dd}{\dd t}\left(\pi\circ \lambda_L(\gamma(t))\right)=\rho(\gamma(t))-\frac{\dd}{\dd t}\left(\tau\circ\gamma(t)\right)=0.$$

\paragraph{Variational interpretation of EL equations.}
Consider a natural action functional $S_L(\cdot)$ associated with a Lagrangian function $L:E\ra \R$: 
$$\gamma:[t_0,t_1]\ra E\quad\quad \longmapsto \quad\quad S_L(\gamma):=\int_{t_0}^{t_1}L(\gamma(t))\dd t.$$
The first variation of $S_L(\cdot)$ at $\gamma$ in the direction of an infinitesimal variation $\del\gamma$ is defined as
$$\<\dd S_L(\gamma),\del\gamma>:=\int_{t_0}^{t_1}\<\dd L(\gamma(t)),\del\gamma(t)>\dd t=\int_{t_0}^{t_1}\dd_TL(\del\gamma(t))\dd t.$$
The following result is according to Grabowska and Grabowski \cite{GG}.

\begin{thm}\label{thm:lagr_var_princ} For an admissible curve $\gamma:[t_0,t_1]\ra E$ and every infinitesimal admissible variation $\del_\xi\gamma$, the following equality holds:
\begin{equation}\label{eqn:EL_var}
\<\dd S_L(\gamma),\del_\xi\gamma>=\int_{t_0}^{t_1}\<\xi(t),\VV_{\Lambda_L(\gamma(t))-\frac{\dd}{\dd t} \lambda_L(\gamma(t))}>\dd t+\<\xi(t),\lambda_L(\gamma(t))>\bigg|^{t_1}_{t_0}.
\end{equation}
Here $\VV_{\Lambda_L(\gamma(t))-\frac{\dd}{\dd t} \lambda_L(\gamma(t))}\in E^\ast$ denotes the vertical part of the vertical vector $\Lambda_L(\gamma(t))-\frac{\dd}{\dd t} \lambda_L(\gamma(t)) \in \V E^\ast\subset\T E^\ast$, i.e. its image under canonical identification $\V E^\ast\approx E^\ast$. 
\end{thm}

\begin{remark}
\label{rem:lagr_var_princ} Observe that \eqref{eqn:EL_var} implies that for an admissible curve $\gamma:[t_0,t_1]\ra E$ the following conditions are equivalent:
\begin{enumerate}
\item The curve $\gamma$ is a solution of the Euler-Lagrange equation for the Lagrangian system $L:E\ra\R$ on the skew-algebroid $\tau:E\ra M$,
\item For every infinitesimal admissible variation $\del_\xi\gamma$, we have 
$$\<\dd S_L(\gamma),\del_\xi\gamma>=\<\xi(t),\lambda_L(\gamma(t))>\bigg|^{t_1}_{t_0}\,,$$
\item For every infinitesimal admissible variation $\del_\xi\gamma$ with vanishing end-points, we have 
$$\<\dd S_L(\gamma),\del_\xi\gamma>=0.$$
\end{enumerate} 
\end{remark}

\begin{remark}\label{rem:var_formal}
The above result can be understood as a variational formulation of the EL equations. Note, however, that it is in some sense different form the ''natural'' variational formulation. In the ''natural'' approach an admissible curve $\gamma(t)$ is an extremal of the action $S_L(\cdot)$ if for every 1-parameter family (homotopy, variation) of admissible curves $t\mapsto\gamma(t,s)$ such that $\gamma(t,0)=\gamma(t)$ (and perhaps satisfying additional conditions such as fixing the end-points) a point $s=0$ is a critical value of the function $s\mapsto S_L(\gamma(\cdot,s))$, i.e.  
$$\<\dd S_L(\gamma),\pa_s\big|_0\gamma(\cdot,s)>=0.$$

In Theorem \ref{thm:lagr_var_princ} and Remark \ref{rem:lagr_var_princ} we present however the ''formal'' approach -- we assume vanishing of $\dd S_L(\gamma)$ on infinitesimal variations of $\gamma$ of certain type ($\del_\xi\gamma$ -- infinitesimal admissible variations with vanishing end-points) without assuring ourselves that these variations are generated by homotopies. Of course from Lemma \ref{lem:adm_var} we know that if $E$ is an AL algebroid then every infinitesimal admissible variation is generated by a homotopy, and hence the ''natural'' and the ''formal'' approach coincide.
In general (if $E$ is not AL) the geometric EL equations \eqref{eqn:EL} are only ''formally'' variational. 
\end{remark}

\paragraph{JVFs.} Now we will give the most important definition of this paper. To motivate it, consider an $s$-parameterized family of solutions of \eqref{eqn:EL}  $t\mapsto\gamma(t,s)$. Then,
$$\frac{\dd}{\dd s}\bigg|_0\frac{\dd}{\dd t}\lambda_L(\gamma(t,s))=\frac{\dd }{\dd s}\Lambda_L(\gamma(t,s))=\T\Lambda_L(\pa_s\big|_0\gamma(t,s)).$$
Now, after using $\kappa_{E^\ast}$ to inverse the order of derivatives on the left-hand side, we get
$$\kappa_{E^\ast}\circ \T\Lambda_L(\pa_s\big|_0\gamma(t,s))=\frac{\dd}{\dd t}\frac{\dd}{\dd s}\bigg|_0\lambda_L(\gamma(t,s))=\frac{\dd}{\dd t}\left(\T\lambda_L(\pa_s\big|_0\gamma(t,s))\right),$$
that is $J(t):=\pa_s\big|_0\gamma(t,s)$ satisfies 
\begin{equation}\label{eqn:Jacobi}
 \frac{\dd}{\dd t}\T\lambda_L(J(t))=\kappa_{E^\ast}\circ\T\Lambda_L(J(t)).
\end{equation}
In other words, the above equation is the tangent prolongation of \eqref{eqn:EL}. Obviously, it is an implicit linear differential equation of the first order. It is tempting to call \eqref{eqn:Jacobi} Jacobi equation, yet this notion would not agree with \eqref{eqn:Jacobi_geod}. The reason is that in the classical geodesic problem we treat the geodesic equation not as the equation for an admissible curve $\gamma=\dot\zeta\in\T M$ but as an equation for the base curve $\zeta\in M$. Therefore we do not look for arbitrary variations of $\gamma=\dot\zeta$ but only for admissible ones (i.e. variations generated by the true variations on the base). This leads us to the following formulation.

\begin{definition}\label{def:Jacobi}
Let $\gamma(t)$ be a solution of the EL equation \eqref{eqn:EL} for a given Lagrangian system $L:E\ra\R$  on a skew-symmetric algebroid $E$. A \emph{Jacobi vector field} (\emph{JVF}) along $\gamma(t)$ is a curve $\eta(t)\in E_{\tau\circ\gamma(t)}$ such that the associated admissible variation $\del_\eta\gamma$ satisfies \eqref{eqn:Jacobi}, i.e. 
\begin{equation}\label{eqn:Jacobi_2}
 \frac{\dd}{\dd t}\T\lambda_L(\del_{\eta(t)}\gamma(t))=\kappa_{E^\ast}\circ\T\Lambda_L(\del_{\eta(t)}\gamma(t)).
\end{equation}
Equation \eqref{eqn:Jacobi_2} will be called the \emph{Jacobi equation}.

We say that the points $\gamma(t_0)$ and $\gamma(t_1)$ are \emph{conjugated along $\gamma$} if there exists a non-zero JVF $\eta(t)\in E_{\tau\circ\gamma(t)} $ such that $\eta(t_0)=\eta(t_1)=0$. Obviously, such JVFs form a vector space. Its dimension will be called the \emph{multiplicity} of conjugate points $\gamma(t_0)$ and $\gamma(t_1)$.
\end{definition}

Observe that, since $\eta\mapsto\del_\eta\gamma$ is linear and \eqref{eqn:Jacobi} is linear in $J$, \eqref{eqn:Jacobi_2} is a second order linear implicit differential equation for $\eta$. To see that \eqref{eqn:Jacobi_2} coincides with \eqref{eqn:Jacobi_geod} in the Riemannian case see Section \ref{sec:examples}.

\begin{remark}\label{rem:Jacobi_adm} Let us now discuss the admissibility of solutions of \eqref{eqn:Jacobi} and \eqref{eqn:Jacobi_2}.  Since every solution of \eqref{eqn:EL} was admissible, every solution $J\in\T_\gamma E$ of its tangent prolongation \eqref{eqn:Jacobi} belongs to $\T_\gamma\Adm([t_0,t_1],E)$. Indeed, an application of $\T\T\pi$ to \eqref{eqn:Jacobi} gives
\begin{equation}\label{eqn:adm_prolongation}
\frac{\dd}{\dd t}\left(\T\tau\circ J(t)\right)=\kappa_M\circ\T\rho(J(T))=\dT\rho(J(t)), 
\end{equation}
which is the tangent prolongation of \eqref{eqn:adm} and hence the condition for $J\in\T_\gamma E$ to belong to $\T_\gamma\Adm([t_0,t_1],E)$. 

Now, if $E$ is AL, then (Lemma \ref{lem:adm_var}) every solution of \eqref{eqn:Jacobi} is of the form $J=\del_\eta\gamma$. In general, if $E$ is not AL, then there can be no solutions of \eqref{eqn:Jacobi} of the form $J=\del_\eta\gamma$. Indeed, recall that $\T\tau\circ\del_\eta\gamma=\rho(\eta)$, hence $\frac{\dd}{\dd t}\left(\T\tau\circ\del_\eta\gamma\right)=\T\rho(\dot\eta)$. Consequently, 
$$\T\rho(\dot\eta)=\frac{\dd}{\dd t}\left(\T\tau\circ\del_\eta\gamma\right)=\kappa_M\circ\T\rho(\del_\eta\gamma)=\kappa_M\circ\T\rho(\kappa_\gamma(\dot\eta))$$
(i.e. $J=\del_\eta\gamma$ satisfies \eqref{eqn:adm_prolongation}) if and only if $$\kappa_M\circ \T\rho(\dot\eta)=\T\rho\circ\kappa(\dot\eta).$$ 
The condition $\kappa_M\circ \T\rho=\T\rho\circ\kappa$  means that $\T\rho$ maps $\kappa$ to $\kappa_M$, that is $\rho$ is a morphism of algebroids $E$ and $\T M$ (in other words, $E$ is AL). 
\end{remark}

\paragraph{The main result.}
It turns out that an admissible variation associated with a JVF is again a solution of the EL equation.
\begin{thm}\label{thm:Jacobi}
Equation \eqref{eqn:Jacobi} is an Euler-Lagrange equation corresponding to the Lagrangian $\dd_\T L:\T E\ra\R$ and the tangent lift algebroid structure $\dd_\T E$. Consequently, if $\eta$ is a JVF along $\gamma$, then the associated admissible variation $\del_\eta\gamma$ is a solution of this EL equation. 
\end{thm}
\begin{proof}
Euler Lagrange equations can be constructed via a well-defined geometric procedure as described on diagram \eqref{diag:EL}. Since, by Theorem \ref{thm:alg_lift}, we know how to describe the structure of the tangent lift algebroid $\dd_\T E$, the geometric construction of the Euler-Lagrange dynamics for $\dd_\T L$ on $\dd_\T E$ is the following:
$$\xymatrix{
\T^\ast\T\ar[rr]^{\dd_\T\eps} E && \T\T E^\ast\\
\T\T^\ast E\ar[u]_{\alpha_E}\ar[rr]^(.4){\T\eps} \ar[d]^{\T\tau_E} && \T\T E^\ast\ar[u]_{\kappa_{E^\ast}}\ar[d]^{\T\tau_{E^\ast}}\\
\T E\ar @/^2pc/ [uu]^{\dd\dd_\T L} \ar@{-->}[uurr]^(.7){\Lambda_{\dd_\T L}} 
\ar @{-->}[rr]^{\lambda_{\dd_\T L}} &&\T E^\ast
}.
$$
On the other hand, \eqref{eqn:Jacobi} can be described as follows:
$$\xymatrix{
\T^\ast\T\ar[rr]^{\dd_\T\eps} E && \T\T E^\ast\\
\T\T^\ast E\ar[u]_{\alpha_E}\ar[rr]^(.4){\T\eps} \ar[d]^{\T\tau_E} && \T\T E^\ast\ar[u]_{\kappa_{E^\ast}}\ar[d]^{\T\tau_{E^\ast}}\\
\T E\ar @/^1pc/[u]^{\T\dd L} \ar @{-->}[uurr] \ar @{-->}[urr]^(.6){\T\Lambda_L} \ar @{-->}[rr]^{T\lambda_L} &&\T E^\ast
}.
$$
To prove the assertion it is enough to check that $\alpha_E\circ(\T\dd L)=\dd (\dd_\T L)$ and $\T \lambda_L=\lambda_{\dd_T L}$. To justify the first equality take $\A\in\T\T E$ represented by a homotopy $\gamma(t,s)\in E$. It is easy to see that 
$$\<\dd(\dT L),\A>_{\tau_{\T E}}=\frac{\dd}{\dd s}\bigg|_0\frac{\dd}{\dd t}\bigg|_0L(\gamma(t,s)).$$
Now,
$$\<\alpha_E\circ \T\dd L,\A>_{\tau_{\T E}}\overset{\eqref{eqn:kappa_alpha}}=\<\T\dd L,\kappa_E(\A)>_{\T\tau_{E}}=\frac{\dd}{\dd t}\bigg|_0\frac{\dd}{\dd s}\bigg|_0L(\gamma(t,s)),$$
since $\kappa_E(\A)$ is represented by $\gamma(s,t)$. As the derivatives commute, 
$$\<\dd(\dT L),\A>_{\tau_{\T E}}=\<\alpha_E\circ \T\dd L,\A>_{\tau_{\T E}}.$$
Finally, observe that 
\begin{align*} 
\T\lambda_L=\T\left(\tau_{E^\ast}\circ\eps\circ\dd L\right)=\T\tau_{E^\ast}\circ\T\eps\circ\T\dd L= \left(\T\tau_{E^\ast}\circ\kappa_{E^\ast}\right)\circ\left(\kappa_{E^\ast}\circ\T\eps\right)\circ\T\dd L\overset{\eqref{eqn:lift_eps}}=\\
\tau_{\T E^\ast}\circ\dT\eps\circ\alpha_E\circ\T\dd L=\tau_{\T E^\ast}\circ\dT\eps\circ\dd\dT L=\lambda_{\dT L}.
\end{align*}
\end{proof}

\section{Second order optimality conditions}\label{sec:variations}

In this section we shall study the relation of JVFs and variational principles. Our main result is Theorem \ref{thm:variations} which shows the link between the presence of conjugate points and null vectors of the bilinear form $\del^2 S_L$ -- the second variation of the action $S_L$. The result is proved almost tautologically thanks to the interpretation of JVFs and Jacobi equation in terms of EL-equation given in Theorem \ref{thm:Jacobi}. Much attention is, however, needed to give a proper definition of the second variation of $S_L$ in the case of a Lagrangian system on an algebroid. Therefore we will start form general considerations about the second variation of a function at its critical point (or a critical point relative to a VB morphism).

Throughout this section we will be working in the AL algebroid setting. It turns out that this is needed for the definition of the second variation to work well. Otherwise we have no control of the vectors tangent to the space of admissible curves in $E$. Moreover, the symmetry of $\del^2 S_L$ requires the Lie algebroid setting (see Lemma \ref{lem:symmetry}). 

\paragraph{What is the second variation of a function?} 
Consider a smooth function
 $F:Q\ra\R$ defined on $Q$ (it can be either a finite-dimensional or a Banach manifold). Naively one defines the second variation of $F$ at $q$ in the directions $X,Y\in\T_qQ$ as
$$\del^2 F(q)(X,Y)=\frac{\dd}{\dd s_1}\bigg|_0\frac{\dd}{\dd s_2}\bigg|_0 F(q(s_1,s_2)),$$
where $q(s_1,s_2)$ is such that $X\in\T_qQ$ is a vector represented by the curve $s_1\mapsto q(s_1,0)$ and $Y\in\T_qQ$ is represented by $s_2\mapsto q(0,s_2)$. In fact, $\frac{\dd}{\dd s_1}\big|_0\frac{\dd}{\dd s_2}\big|_0 F(q(s_1,s_2))=\dT\dT F(\A)$, where $\A\in\T\T Q$ is a vector corresponding to the homotopy $(s_1,s_2)\mapsto q(s_1,s_2)$.  Observe that since $\frac{\dd}{\dd s_1}$ and $\frac{\dd}{\dd s_2}$ commute, then $\dT\dT F(\A)=\dT\dT F(\kappa_Q(\A))$. 

To make sure that the second variation of $F$ at $q$ is well-defined, we should therefore show that given any other $\A^{'}\in\T\T Q$ such that $\tau_{\T Q}(\A^{'})=\tau_{\T Q}(\A)=Y$ and $\T\tau_Q(\A^{'})=\T\tau_Q(\A)=X$ we have 
$$\dT\dT F(\A^{'})=\dT\dT F(\A),$$
i.e. that the value of $\dT\dT F(\A)$ depends on $\T Q$-projections of $\A$ only. But,
$$\dT\dT F(\A)-\dT\dT F(\A^{'})=\<\dd\dT F, \A-\A^{'}>,$$
where we subtract w.r.t. the vector bundle structure $\tau_{\T Q}$ on $\T\T Q$. Since $\A$ and $\A^{'}$ have the same $\T\tau_Q$-projections, $\A-\A^{'}\in\T_Y\T Q$ is vertical w.r.t. $\T\tau_Q$, and hence in the canonical identification $\V\T Q\approx\T Q\times_Q\T Q$ it corresponds to a pair $(Y,Z)$. A simple calculation shows that 
$$\<\dd\dT F,\A-\A^{'}>=\dT F(Z)=\<\dd F(q),Z>.$$
Consequently, $\del^2 F(q)$ is well-defined if and only if $\dd F(q)=0$, i.e. $q$ is a critical point of $F$. 

We can summarize our considerations as follows:
\begin{prop}\label{prop:second_var}
Assume that $q\in Q$ is a critical point of $F$. Then the second variation of $F$ at $q$ in the directions $X,Y\in\T_q Q$ is well-defined via the formula:
\begin{equation}\label{eqn:second_var}
\del^2F(q)(X,Y):=\dT\dT F(\A),
\end{equation}
where $\A\in\T\T_qQ$ is any vector projecting to $X$ and $Y$ under $\T\tau_Q$ and $\tau_{\T Q}$, respectively. 

Moreover the second variation is symmetric, i.e. 
$$\del^2F(q)(X,Y)=\del^2F(q)(Y,X).$$
\end{prop}

\paragraph{The second variation at a relative critical point.}

From the point of view of Variational Calculus it is more natural to consider critical points of a function relative to a vector bundle morphism, rather than just critical points. Our geometric setting is now the following: $F:Q\ra\R$ is a smooth function, $\sigma:S\ra N$ is a vector bundle and, additionally, we have a vector bundle morphism $P:\T Q\ra S$ over $p:Q\ra N$. A point $q\in Q$ is called a \emph{critical point of $F$ relative to $P$} (which will be denoted by $q\in\Crit_P(F)$) if 
$$\<\dd F(q),X>=0\quad\text{for every $X\in\T_q Q\cap \ker P$}.$$

To give a concrete example think about the standard variational problem associated with a Lagrangian $L:\T M\ra\R$. We may take $Q=\Adm([t_0,t_1],\T M)$ to be the space of admissible curves in $\T M$ (i.e. tangent lifts of smooth curves $\zeta:[t_0,t_1]\ra M$); $N=M\times M$, $\sigma=\tau_N:\T N\ra N$, and $P=\T p:\T Q\ra \T N$, where $p$ sends a curve $\dot\zeta(\cdot)$ to its end-points $\left(\zeta(t_0),\zeta(t_1)\right)\in M\times M$. Define  $F(\dot\zeta):=\int_{t_0}^{t_1}L(\dot\zeta(t))\dd t$. Now the solutions of the EL-equation (extremals of the action) are precisely critical points of $F$ relative to $\T p$ ($\dd F$ vanishes on infinitesimal admissible variations with vanishing end-points - cf. Remark \ref{rem:lagr_var_princ} and Theorem \ref{thm:lagr_var_princ}).

Now we would like to define the second variation of $F$ at $q\in\Crit_P(F)$.  Obviously, the idea from Proposition \ref{prop:second_var} will not work directly. Indeed, taking vectors 
$\A, \A^{'}\in\T\T Q$ as in the previous paragraph we have
$$\dT\dT F(\A)-\dT\dT F(\A^{'})=\<\dd F(q),Z>,$$
which is non-zero unless $Z\in\ker P$. To guarantee this we cannot take arbitrary vectors $\A$ and $\A^{'}$ with the same $\T Q$-projections $X$ and $Y$, but have to restrict ourselves to a certain subclass of such vectors. 

It is easy to check that if $\A-\A^{'}\in\T_Y\T Q$ is a vertical vector corresponding to $(Y,Z)$, then $\T P(\A)$ and $\T P(\A^{'})$ project to $P(Y)$ and $\T p(X)$ under $\tau_{S}$ and $\T\sigma$, respectively. The difference $\T P(\A)-\T P(\A^{'})\in \T_{P(Y)}\T N$ is a vertical vector w.r.t. $\T\sigma$ corresponding to $(P(Y),P(Z))$. It follows that $Z\in\ker P$ if and only if $\T P(\A)-\T P(\A^{'})$ corresponds to $(P(Y),0)$. This in turn implies $\T P(\A)=\T P(\A^{'})$. In other words, formula \eqref{eqn:second_var} will give a good definition of the second variation of $F$ at $q\in\Crit_P(F)$ if we consider $\A$'s with fixed $\T P$-projections. Of course $\T P(\A)\in\T S$ should be consistent with $\T Q$-projections of $\A$, i.e. it should project to $P(Y)$ under $\tau_S$ and to $\T p(X)$ under $\T\sigma$. Let us summarize:

\begin{prop}\label{prop:second_var_rel} Assume that  $q\in \Crit_P(F)$. Consider an arbitrary map assigning to a pair of vectors $\ol X\in\T_{p(q)} N$ and $\wt Y\in S_{p(q)}$ a vector $h(\ol X,\wt Y)\in \T S$ such that $\tau_{S}\circ h(\ol X,\wt Y)=\wt Y$ and $\T\sigma\circ h(\ol X,\wt Y)=\ol X$ (i.e. $ h(\ol X,\wt Y)$ is a lift of $\ol X$ to $\T_{\wt Y}S$). Then we may define the second variation of $F$ at $q$ w.r.t. $h$ via the formula:
$$\del^2 F(q,h)(X,Y):=\dT\dT F(\A),$$
where $\A\in\T\T Q$ is an arbitrary vector projecting to $X$ under $\T\tau_Q$, to $Y$ under $\tau_{\T Q}$ and to $h(\T p(X),P(Y))$ under $\T P$.
\end{prop}

In general there is no canonical choice of $h$ unless $\sigma$ is provided with some additional structure. For example given a connection on $\sigma$, we may choose $h(\ol X,\wt Y)$ to be the horizontal lift of the base vector $\ol X$ to $\T_{\wt Y}S$. Observe that if the connection is linear, then  $\del^2 F(q,h)(X,Y)$ becomes a bilinear form. Indeed, to prove linearity w.r.t. $X$ observe that if $\A\in\T\T Q$ projects to $X$, $Y$ and $h(\T p(X),P(Y))$; and $\A^{'}$ to $X^{'}$, $Y$ and $h(\T p(X^{'}),P(Y))$, then the linear combination $\lambda\A+\mu\A^{'}$ (taken w.r.t. $\tau_{\T Q}$) projects to $\lambda X+\mu X^{'}$, $Y$ and $\lambda h(\T p( X),P(Y))+\mu h(\T p(X^{'}),P(Y))= h(\T p(\lambda X+\mu X^{'}),P(Y))$. Similarly, using the VB structure $\T\tau_Q$ on $\T\T Q$ and linearity of $h(\ol X,\wt Y)$ w.r.t. $\wt Y$ (the connection is linear) one can prove linearity of $\del^2 F(q,h)(X,Y)$ w.r.t. $Y$. 

The symmetry of $\del^2 F(q,h)$ may be quite problematic. In general one requires that if $\A\in\T\T Q$ is a vector projecting to $X$, $Y$ and $h(\T p(X),P(Y))$ then $\kappa_Q(\A)$ projects to  $Y$, $X$ and $h(\T p(Y),P(X))$ (i.e. $\kappa_Q$ preserved the distinguished class of vectors in $\T\T Q$). If this is the case then the symmetry of the second variation follows from the $\kappa_Q$-invariance of $\dT\dT F$:
$$ \del^2F(q,h)(X,Y)=\dT\dT F(\A)=\dT\dT F(\kappa_Q(\A))=\del^2F(q,h)(Y,X).$$

Natural examples of such a situation appear when considering symmetric connections. Take a tangent bundle  $\sigma=\tau_N:\T N\ra N$; $p:Q\ra N$ to be any smooth map and $P=\T p:\T Q\ra \T N$. If $h$ is a horizontal lift of a symmetric linear connection on $\tau_N$ then $\kappa_N\circ h(\ol X,\wt Y)=h(\wt Y,\ol X)$ for any $\ol X,\wt Y\in \T_x N$. Now if $\A\in\T\T Q$ projects to $X$, $Y$ and $h(\T p(X),\T p(Y))$ then $\kappa_Q(\A)$ projects to $Y$, $X$ and 
$$\T P(\kappa_Q(\A))=\T\T p(\kappa_Q(\A))=\kappa_N(\T\T p(\A))=\kappa_N(h(\T p(X),\T p(Y)))=h(\T p(Y),\T p(X)),$$
and hence $\del^2 F(q,h)$ is a symmetric bilinear form. More generally, one can check that if $\sigma:S\ra N$ carries an algebroid structure; $P:\T Q\ra S$ is an algebroid morphism and $h$ is a horizontal lift of a symmetric linear connection on $\sigma$ (symmetric w.r.t. to the algebroid structure), then $\del^2 F(q,h)$ is a symmetric bilinear form. 

Without the presence of a connection on $\sigma$ we may define the second variation $\del^2 F(q)(X,Y)$ after restricting our attention to $Y$'s belonging to $\ker P$. Since $P(Y)=0$ we do not have to choose $h(\ol X,\wt Y)$ for any $\ol X$ and $\wt Y$, but it suffices to define $h(\ol X,0)$ only. There is the canonical choice $h(\ol X,0)=\theta_{\ol X}$ -- the null vector w.r.t. $\T\sigma$ at $\ol X\in\T N$. Since $\tau_S$-addition gives $\theta_{\ol X}+\theta_{\ol X^{'}}=\theta_{\ol X+\ol X^{'}}$ and $\T \sigma$-addition gives  $\theta_{\ol X}+\theta_{\ol X}=\theta_{\ol X}$, the second variation will be a bilinear form for this choice of $h$. 

Of course in this setting there is no sense speaking about the symmetry of the second variation unless we restrict ourselves to $X,Y\in\T_qQ\cap \ker P$. It seems that in general the second variation will not be symmetric, without some additional assumptions on $P$ and $\sigma$.  To sum up:

\begin{lem}\label{lem:second_var_rel}
Let $q\in Q$ be a critical point of $F$ relative  to $P:\T Q\ra S$. Then for $X\in \T_q Q$ and $Y\in\T_q Q\cap\ker P$ we may define the \emph{second variation of $F$ at $q$ in the directions $X,Y$} by
$$\del^2 F(q)(X,Y):=\dT\dT F(\A),$$
where $\A\in\T\T Q$ is an arbitrary vector projecting to $X$ under $\T\tau_Q$, to $Y$ under $\tau_{\T Q}$ and to $\theta_{\T p(X)}$ under $\T P$.

Moreover, $\del^2 F(q)(\cdot,\cdot)$ is a bilinear form. 
\end{lem}

\paragraph{The second variation in variational problems.}
We would like to apply the construction from Lemma \ref{lem:second_var_rel} to variational problems on AL algebroids. In this case we can consider  $Q=\Adm([t_0,t_1],E)$ -- the Banach manifold of admissible paths $\gamma:[t_0,t_1]\ra E$ and $F=S_L$ to be the action associated to  a Lagrangian system $L:E\ra M$ 
$$S_L(\gamma):=\int_{t_0}^{t_1}L(\gamma(t))\dd t.$$

Since $E$ is AL, the tangent space of $Q$ at $\gamma$ consists of admissible variations $\T_\gamma Q=\{\del_\xi\gamma:\xi\in E_{\tau\circ\gamma}\}$. Assume that $\xi\mapsto \del_\xi\gamma$ is an isomorphism of vector spaces  $E_{\tau\circ\gamma}$ and $\T_\gamma Q$. (In general $\xi\mapsto\del_\xi\gamma$ is just a linear epimorphism. In such a case a little more attention is needed to construct the second variation of $S_L$ -- see Remark \ref{rem:non_iso}). An admissible curve $\gamma$ is a solution of the EL equation \eqref{eqn:EL} if and only if $\gamma$ is a critical point of $S_L$ relative to the map $P:\T Q\ra E\times E$, which sends $\del_\xi\gamma$ to $(\xi(t_0),\xi(t_1))$. Indeed, $\T_\gamma Q\cap \ker P$ is precisely the set of admissible variations with vanishing end-points. 

Our goal is to define $\del^2 S_L(\gamma)(\del_\eta\gamma,\del_\xi\gamma)$ -- the second variation at $\gamma$ (which is a solution of the EL equation) in the directions of $\del_\eta\gamma$ and $\del_\xi\gamma$. For simplicity we will use the abbreviated notation $\del^2 S_L(\gamma)(\eta,\xi)$ and will treat the second variation as a function of generators $\eta$ and $\xi$. 

To proceed within the scheme sketched in the previous paragraph (in Lemma \ref{lem:second_var_rel}) we need more information about the iterated tangent space $\T\T Q$. To describe $\T_{\del_\xi\gamma}\T Q$ consider an s-parameterized family of admissible variations $\del_{\xi(t,s)}\gamma(t,s)$ around $\del_{\xi(t)}\gamma(t)=\del_{\xi(t,0)}\gamma(t,0)$. Denote by $\Delta$ the differentiation $\frac{\pa}{\pa s}$ at $s=0$. Now
\begin{align*}\Delta(\del_{\xi(t,s)}\gamma(t,s))&=\Delta\left(\kappa_{\gamma(t,s)}(\dot\xi(t,s))\right)=\T\kappa_{\Delta\gamma}\left(\Delta(\dot\xi)\right)=\T\kappa_{\Delta\gamma}(\kappa_E(\Delta\xi)^{\cdot})=\\
&=\kappa_E\circ\left(\kappa_E\circ\T\kappa_{\Delta\gamma}\circ\kappa_E\right)\left((\Delta\xi)^{\cdot}\right)=\kappa_E\circ\dT\kappa_{\Delta\gamma}\left((\Delta\xi)^{\cdot}\right), \end{align*}
where $\Delta\gamma(t)=\frac{\pa}{\pa s}\big|_{0}\gamma(t,s)\in\T_\gamma E$ and  $\Delta\xi(t)=\frac{\pa}{\pa s}\big|_{0}\xi(t,s)\in\T_\xi E$. Observe that  $\Delta\gamma$ and $\Delta\xi$ project to the same path $\Delta x(t):=\T\tau\circ\Delta\xi(t)=\T\tau\circ\Delta\gamma(t)=\frac{\pa}{\pa s}\big|_{0}x(t,s)$ and that $\Delta\gamma$ is $\dT E$-admissible, i.e. 
$$\left(\T\tau\circ\Delta\gamma\right)^{\cdot}=(\Delta x)^{\cdot}=\kappa_M\circ\Delta(\dot x)=\kappa_M\circ\Delta(\rho(\gamma))=\kappa_M\circ\T\rho(\Delta\gamma)=\dT\kappa(\Delta\gamma).$$
If follows that $\dT\kappa_{\Delta\gamma}\left(\Delta\xi\right)^{\cdot}$ is a $\dT E$-admissible variation along a $\dT E$-admissible path $\Delta\gamma$ with a generator $\Delta\xi$. With some abuse of notation we will denote it by $\del_{\Delta\xi}\Delta\gamma$. Since we deal with an AL algebroid and $\gamma(t,s)$ was a family of admissible paths, from Lemma \ref{lem:adm_var} we have $\Delta \gamma=\del_\eta\gamma$ and $\Delta x=\rho(\eta)$ for some $\eta\in E_{\tau\circ\gamma}$. 

We see that elements of $\T_{\del_\xi\gamma}\T Q$ are generated by elements 
$\Delta\xi\in \T_\xi E$ and $\eta\in E_{\tau\circ\gamma}$ satisfying the compatibility condition $\T\tau\circ\Delta\xi=\rho(\eta)$. Such an element will be denoted by $\del_{(\Delta\xi,\eta)}\gamma$. The precise formula is 
$$\del_{(\Delta\xi,\eta)}\gamma:=\kappa_E\circ\del_{\Delta\xi}(\del_\eta\gamma).$$
By construction $\del_{(\Delta\xi,\eta)}\gamma$ projects to $\del_\xi\gamma$ under $\tau_{\T E}$ and to $\del_\eta\gamma$ under $\T\tau_E$. Let us note that since $P(\del_\xi\gamma)=(\xi(t_0),\xi(t_1))$, then $\T P\left( \del_{(\Delta\xi,\eta)}\gamma\right)=\left(\Delta\xi(t_0),\Delta\xi(t_1)\right)$. Due to Lemma \ref{lem:second_var_rel} we may define the second variation of $S_L$ at $\gamma$ after restricting our attention to $\tau_{\T E}\circ\del_{(\Delta\xi,\eta)}\gamma=\del_\xi\gamma\in\ker P$, i.e. $\xi\in E^0_{\tau\circ\gamma}$ and  $\T P\left( \del_{(\Delta\xi,\eta)}\gamma\right)\in\ker\T P$ i.e. $\Delta\xi(t_0)$ and $\Delta\xi(t_1)$ being null vectors w.r.t. $\T\tau:\T E\ra \T M$. To sum up we can state the following lemma:

\begin{lem}\label{lem:second_adm_var}
Let $\gamma$ be a solution of the EL equation \eqref{eqn:EL}. 
Formula
\begin{equation}\label{eqn:delta2_L}
\del^2 S_L(\gamma)(\eta,\xi):=\dT^2S_L(\del_{(\Delta\xi,\eta)}\gamma)=\int_{t_0}^{t_1}\<\dd\dT L(\del_\xi\gamma), \del_{(\Delta\xi,\eta)}\gamma(t)>\dd t;
\end{equation}
where $\Delta\xi\in\T_\xi E$ is any curve such that $\T\tau\circ\Delta\xi=\T\tau\circ\Delta\gamma=\rho(\eta)$ and $\Delta\xi(t_0)$, $\Delta\xi(t_1)$ are null vectors w.r.t.  $\T\tau:\T E\ra \T M$; defines a bilinear form $\del^2 S_L(\gamma)(\cdot,\cdot):E_{\tau\circ\gamma}\times E^0_{\tau\circ\gamma}\ra \R$. We will call it the \emph{second variation of $S_L$ at $\gamma$}. 

Moreover the following equality holds
\begin{equation}\label{eqn:delta2_L_eqn}
\del^2 S_L(\gamma)(\eta,\xi):=\int_{t_0}^{t_1}\<\Delta\xi(t),\VV_{\kappa_{E^\ast}\circ\T\Lambda_{L}(\del_\eta\gamma(t))-\frac{\dd}{\dd t} \T\lambda_{L}(\del_\eta\gamma(t))}>_{\T\tau}\dd t.
\end{equation}
\end{lem}
  
\begin{proof}
The fact that the second variation is well-defined by formula \eqref{eqn:delta2_L} follows from Lemma \ref{lem:second_var_rel}. We need to check that equation \eqref{eqn:delta2_L_eqn} holds. First of all using the $\kappa_E$-invariance of $\dT^2S_L$ we get
$$
\del^2 S_L(\gamma)(\eta,\xi)=\dT\dT S_L\left(\del_{(\Delta\xi,\eta)}\gamma\right)=\dT\dT S_L\left(\kappa_E\circ\del_{\Delta\xi}(\del_\eta\gamma)\right)=
\dT\dT S_L\left(\delta_{\Delta \xi}(\del_\eta\gamma)\right).
$$
Since $\dT S_L=S_{\dT L}$, the later equals $\dT S_{\dT L}\left(\delta_{\Delta \xi}(\del_\eta\gamma)\right)$. We can use \eqref{eqn:EL_var} for a Lagrangian system $\dT E:\T E\ra \R$ on $\dT E$ to get
\begin{align*}
\del^2 S_L(\gamma)(\eta,\xi)=&\dT S_{\dT L}(\delta_{\Delta \xi}\Delta\gamma)=\<\dd S_{\dT L}(\del_\eta\gamma),\delta_{\Delta \xi}(\del_\eta\gamma)>=\\
&\int_{t_0}^{t_1}\<\Delta\xi(t),\VV_{\Lambda_{\dT L}(\del_\eta\gamma(t))-\frac{\dd}{\dd t} \lambda_{\dT L}(\del_\eta\gamma(t))}>_{\T\tau}\dd t+\<\Delta\xi(t),\lambda_{\dT L}(\del_\eta\gamma(t))>_{\T\tau}\bigg|^{t_1}_{t_0}.
\end{align*}
The last summand vanishes as $\Delta \xi(t)$ vanishes at the end-points.
Finally note that from the proof of Theorem \ref{thm:Jacobi} we know that $\Lambda_{\dT L}=\kappa_{E^\ast}\circ \T\Lambda_L$ and that $\lambda_{\dT L}=\T\lambda_L$. This finishes the proof of \eqref{eqn:delta2_L_eqn}. 
\end{proof}

Formula \eqref{eqn:delta2_L_eqn} generalizes the classical formula \eqref{eqn:C} -- it relates the second variation of $S_L$ to Jacobi equation \eqref{eqn:Jacobi_2}. The details are provided in Section \ref{ssec:riem_geom}.

\begin{remark}\label{rem:non_iso}
Note that the variational interpretation of the EL equation on algebroids (Remark \ref{rem:lagr_var_princ}) is slightly more general than what we considered in the paragraph about relative critical points. Solutions of the EL equation are critical points of the map
\begin{align*}
\Map([t_0,t_1],E)&\lra\T Q\overset{\dd F}\lra\R\\
\xi&\longmapsto \del_\xi\gamma\longmapsto\dd F(\del_\xi\gamma)
\end{align*}
relative to the map $P:\Map([t_0,t_1],E)\ni \xi\longmapsto(\xi(t_0),\xi(t_1))\in E\times E$ rather than critical points of $\dd F:\T Q\ra\R$ relative to $P:\T Q\ra E\times E$. One can say that we do not work directly with $\T Q$ but we do it via the map $\xi\mapsto \del_\xi\gamma$. 

If  $\T_\gamma Q$ and $E_{\tau\circ\gamma}$ are isomorphic (as we assumed in this paragraph), then the above setting coincides with the setting considered in the previous paragraph. However, the results and constructions of this paragraph hold even if the two spaces are not isomorphic. The reason is that what one really needs is the image of $\ker P$ in $\T Q$ under $\xi\mapsto\del_\xi\gamma$ and the image of $\ker \T P$ in $\T\T Q$ under $(\Delta\xi,\eta)\mapsto\del_{(\Delta\xi,\eta)}\gamma$. 

Finally let us comment on the isomorphism between $E_{\tau\circ\gamma}$ and $\T_\gamma Q$. It was proved in \cite{GJ_PMP} that, if $E$ is an AL algebroid, then given an element $\del\gamma(t)\in\T_\gamma Q$ and $\xi_0\in E_{\tau\circ\gamma(t_0)}$ such that $\rho(\xi_0)=\T\tau(\del\gamma(t_0))$ there exists a unique $\xi(t)\in E_{\tau\circ\gamma}$ such that $\del_\xi\gamma=\del\gamma$ and $\xi(t_0)=\xi_0$ (to find such a $\xi$ one needs to solve a linear ODE). If $E$ is not transitive (i.e. $\ker\rho$ is nontrivial) there may be several elements $\xi_0$ and several corresponding curves $\xi(t)$ for which $\del\gamma=\del_\xi\gamma$. In such a case $\xi\mapsto\del_\xi\gamma$ is just an epimorphism. 
\end{remark}

\paragraph{Symmetry of the second variation.} 

The question if $\del^2 S_L(\gamma)$ is a symmetric bilinear form needs some attention. First of all to speak about symmetry we have to restrict our attention to pairs of generators vanishing at the end-points, i.e. consider $\del^2 S_L(\gamma)(\cdot,\cdot):E^0_{\tau\circ\gamma}\times E^0_{\tau\circ\gamma}\ra \R$. 

Recall that the second variation $\del^2 S_L(\gamma)(\eta,\xi)$ is defined by means of a vector $\del_{(\Delta\xi,\eta)}\gamma$ with generators $\Delta\xi$ and $\eta$ vanishing at the end-points. To define $\del^2 S_L(\gamma)(\xi,\eta)$ we need another vector $\del_{(\Delta\eta,\xi)}\gamma$ with $\Delta\eta$ and $\xi$ vanishing at the end-points. For a given $(\Delta\xi,\eta)$ there is a natural choice of $\Delta\eta$ - we take a unique vector in $\T_\eta E$ $\kappa$-related to $\Delta\xi$. The comparison of vectors $\del_{(\Delta\xi,\eta)}\gamma$ and $\del_{(\Delta\eta,\xi)}\gamma$ can answer the question of symmetry. The conclusion is rather surprising: we can guarantee the symmetry of the second variation only for Lie algebroids. Simple examples form Section \ref{ssec:riem_geom} show that this condition cannot be omitted.

\begin{lem}\label{lem:symmetry}
Let $\gamma$ be the solution of the EL equation \eqref{eqn:EL} on an AL algebroid $E$. If $E$ is a Lie algebroid then the second variation $\del^2 S_L(\gamma)$ is symmetric after restricting to $E^0_{\tau\circ\gamma}\times E^0_{\tau\circ\gamma}$.
\end{lem}

\begin{proof}
Take any $\xi,\eta\in E^0_{\tau\circ\gamma}$ and $\Delta \xi\in\T_\xi E$ vanishing at the end-points. Denote by $\Delta\eta\in\T_{\eta}E$ a unique vector field along $\eta$ which is $\kappa$-related to $\Delta\xi$. Observe that since $\kappa$ is a linear relation, then $\Delta\eta$ also vanishes at the end-points. Now
$$\del^2S_L(\gamma)(\eta,\xi)=\int_{t_0}^{t_1}\dT^2L(\del_{(\Delta\xi,\eta)}\gamma)\dd t$$
and
$$\del^2S_L(\gamma)(\xi,\eta)=\int_{t_0}^{t_1}\dT^2L(\del_{(\Delta\eta,\xi)}\gamma)\dd t.$$
Since $\dT^2L$ is $\kappa_E$-invariant the later equals $\int_{t_0}^{t_1}\dT^2L(\kappa_E\circ\del_{(\Delta\eta,\xi)}\gamma)\dd t$. Our goal is to compare vectors $\del_{(\Delta\xi,\eta)}\gamma$ and $\kappa_E\circ\del_{(\Delta\eta,\xi)}\gamma=\del_{\Delta\eta}(\del_\xi\gamma)$. They both belong to $\T_{\del_\xi\gamma}\T E$ and both project to $\del_\eta\gamma$ under $\T\tau_E$. We claim that if $E$ is a Lie algebroid then
\begin{equation}\label{eqn:del_equal}
\del_{(\Delta\xi,\eta)}\gamma=\del_{(\Delta\eta,\xi)}\gamma.
\end{equation}
Above equality has a simple coordinate proof. The reader is invited to do the necessary calculation using formula \eqref{eqn:del2_coor} form the next section. A short calculation shows that the difference $\del_{(\Delta\xi,\eta)}\gamma-\del_{(\Delta\eta,\xi)}\gamma$ contains the RHSs of equations \eqref{eqn:ala_coor} and \eqref{eqn:lie_coor} which are local versions of \eqref{eqn:ala} and \eqref{eqn:lie}. To get the equality these conditions have to be satisfied. In fact this difference is a vector in $\T\T E$ vertical w.r.t. $\T\tau_E$ which corresponds to $(\del_\xi\gamma, J(\gamma,\eta,\xi))$ in the canonical identification $\V\T E\approx\T E\times_E\T E$. Here $J(\cdot,\cdot,\cdot)$ denotes the Jacobiator.

The geometric interpretation of the above calculations is the following. According to \cite{GU} $E$ carries the structure of a Lie algebroid if and only if the associated DVB morphism $\eps:\T^\ast E\ra\T E^\ast$ maps the canonical Poisson structure on $\T^\ast E$ to the bivector field $\dT\Pi$ associated with the algebroid structure $\dT E$, i.e. $\eps$ is a morphism of algebroid structures. This condition can be also expressed in the language of relations $\kappa$. Roughly speaking $\T\kappa$ should intertwine relation $\dT\kappa$ and $\kappa_E$. This is the sense of equation \eqref{eqn:del_equal}. Some more work is needed to make this intuition rigorous.   
\end{proof}

\paragraph{JVFs and the second variations.} Now we are ready to formulate the main result of this section. 

\begin{thm}\label{thm:variations} Let $L:E\ra\R$ be a Lagrangian system on an AL algebroid $\tau:E\ra M$ and
let $\gamma$ be a solution of the EL equation \eqref{eqn:EL}. A curve $\eta\in E_{\tau\circ\gamma}$ is a JVF along $\gamma$ if and only if 
$$\del^2S_L(\gamma)(\eta,\xi)=0\quad\text{for every $\xi\in E^0_{\tau\circ\gamma}$.}$$

If $E$ is a Lie algebroid then the null space of the symmetric bilinear form
$$\del^2S_L(\gamma)(\cdot,\cdot):E^0_{\tau\circ\gamma}\times E^0_{\tau\circ\gamma}\lra\R$$
consists of all JVFs along $\gamma$ vanishing at the end-points $t_0$ and $t_1$  (i.e. $\gamma(t_0)$ and $\gamma(t_1)$ are conjugated along $\gamma$ if the null space is non-trivial). The multiplicity of $\gamma(t_0)$ and $\gamma(t_1)$ as conjugate points equals the dimension of this null space.
\end{thm}

\begin{proof}
The proof is straightforward. We know that $\eta$ is a JVF if and only if $\del_\eta\gamma$ satisfies the EL-equation for the Lagrangian $\dT L:\T E\ra \R$ on the lifted algebroid $\dT E$ (cf. Theorem \ref{thm:Jacobi}). From Remark \ref{rem:lagr_var_princ} the later is equivalent to the vanishing of 
$$\dT S_{\dT L}(\del_{\Delta\xi}(\del_\eta\gamma))=\dT\dT S_L(\del_{\Delta\xi}(\del_\eta\gamma))=\dT^2 S_L(\kappa_E\circ\del_{\Delta\xi}(\del_\eta\gamma))=\dT^2S_L((\del_{(\Delta\xi,\eta)}\gamma))$$
for every $\dT E$-admissible variation $\del_{\Delta\xi}(\del_\eta\gamma)$ with vanishing end-points. From \eqref{eqn:delta2_L} we know that this is equivalent to the vanishing of $\del^2 S_L(\gamma)(\eta,\xi)$ for every $\xi\in E^0_{\tau\circ\gamma}$. 

The second part of the assertion is trivial. From Lemma \ref{lem:symmetry} we know that $\del^2 S_L(\gamma)(\cdot,\cdot)$ is symmetric for Lie algebroids and that, by definition, existence and multiplicity of conjugated points along $\gamma$ is defined by means of JVFs along $\gamma$ vanishing at the end-points. 
\end{proof}

Observe that in general we cannot expect the dimension of the null space of $\del^2S_L(\gamma)$ to be finite (as is in the Riemannian case). Note namely, that Jacobi equation \eqref{eqn:Jacobi_2} is an implicit linear differential equation, so we cannot estimate the number of independent solutions. In particular if $L=0$, then ever curve $\xi\in E_{\tau\circ\gamma}$ is a JVF. However under extra assumptions for $L$ it is possible to control the number of solutions and consequently, the dimension of the null space. 

\section{Coordinate calculations}\label{sec:coor}

So far we worked in purely geometrical terms. In this section we give local description of the most important results and formulas. 

\paragraph{Manifolds and vector bundles}
We will be using the same notation conventions as before; i.e. $M$, $Q$ are manifolds, $\tau:E\ra M$ is a vector bundle $\pi:E^\ast\ra M$ its dual, etc. 

On $Q$ we introduce local coordinates $(q^\alpha)\ \alpha=1,\hdots,m$. We have induced coordinates $(q^\alpha,\dot q^\beta)$ on $\T Q$ and $(q^\alpha,p_\beta)$ on $\T^\ast Q$. On the iterated tangent bundle  $\T\T Q$ we have induced coordinates $(q^\alpha,\dot q^\beta,\del q^\gamma, \del \dot q^\delta)$;
$(q^\alpha,p_\beta,\dot q^\gamma,  \dot p_\delta)$ on $\T\T^\ast Q$ and $(q^\alpha,\dot q^\beta,p_\gamma, \pi_\delta)$ on $\T^\ast\T Q$.

On $\tau: E \rightarrow M$ we introduce local coordinates $(x^a, y^i)$,\ $a=1,\hdots, n$; $i=1,\hdots, k$, where $(x^a)$ are local base coordinates and $(y^i)$ are linear coordinates on fibers. The choice of such coordinates is equivalent to the choice of a basis of local sections $(e_1,\hdots,e_k)$ of $\tau:
E\rightarrow M$. On the dual bundle $\pi:E^\ast\ra M$ we have dual coordinates 
$(x^a, z_i)$.

After choosing the basis  $(e_1,\hdots,e_k)$ we can introduce natural local coordinates on tangent and cotangent bundles of $E$ and $E^\ast$
\begin{align*}
    &(x^a, y^i,{\dot x}^b, {\dot y}^j )  \text{ in $\T E$},
  &&(x^a, z_i, {\dot x}^b, \dot{z}_j) \text{ in $\T E^\ast$},\\
  &(x^a, y^i, p_b, \pi_j) \text{ in $\T^\ast E$},
    &&(x^a, \xi_i, p_b, \varphi^j) \text{ in $\T^\ast E^\ast $} .
\end{align*}

\paragraph{Algebroid structure.}

The skew-algebroid structure on $E$ is defined by means of functions $\rho^a_i(x)$ and $c^i_{jk}(x)$ given by 
$$\rho(e_i)=\rho^a_i(x)\pa_{x^a}$$
and
$$[e_j,e_j]=e_i\cdot c^i_{jk}(x).$$
The skew-symmetry of $[\cdot,\cdot]$ reads as $c^i_{jk}(x)=-c^i_{kj}(x)$; condition \eqref{eqn:ala} for $E$ being AL takes the form
\begin{equation}\label{eqn:ala_coor}
\rho^a_k(x)c^k_{ij}(x)-\left(\frac{\pa \rho^a_j}{\pa x^b}(x)\rho^b_i(x)-\frac{\pa\rho^a_i}{\pa x^b}(x)\rho^b_j(x)\right)=0
\end{equation}
ant condition \eqref{eqn:lie} for $E$ being Lie is \eqref{eqn:ala_coor} together with
\begin{equation}\label{eqn:lie_coor}
c^s_{lk}(x)c^l_{ij}(x)+c^s_{li}(x)c^l_{jk}(x)+c^s_{lj}(x)c^l_{ki}(x)-\frac{\pa c^s_{ij}}{\pa x^a}(x)\rho^a_k(x)-\frac{\pa c^s_{jk}}{\pa x^a}(x)\rho^a_i(x)-\frac{\pa c^s_{ki}}{\pa x^a}(x)\rho^a_k(x)=0.
\end{equation}
The canonical bivector $\Pi\in\X^2 (E^\ast)$ associated with the algebroid structure on $\tau$ reads as 
\begin{equation}\label{eqn:Pi_coor}
\Pi(x,z)=\frac 12c^k_{ij}(x)\xi_k\pa_{z_i}\wedge\pa_{z_j}+\rho^b_i(x)\pa_{x_b}\wedge \pa_{z_i}
\end{equation}
and the induced DVB map $\eps:\T^\ast E\lra\T E^\ast$ as
\begin{equation}
\label{eqn:eps_coor}
\eps(x^a,y^i,p_b,z_j) = (x^a, z_i, \rho^b_k(x)y^k, c^k_{ij}(x) y^iz_k + \rho^a_j(x) p_a).
\end{equation}
Relation $\kappa$ dual to $\eps$ (see \eqref{eqn:kappa}) relates vectors 
\begin{equation}
\label{eqn:kappa_coor}\kappa:\left(x^a,\,{Y}^i,\,\rho^b_k(x)y^k,\,\dot{Y}^j\right)\rel\left(x^a,y^i,\,\sigma^b_k(x)Y^k,\,
\dot{Y}^j+c^j_{kl}(x)y^kY^l\right)\,.
\end{equation}
It follows that if $\gamma(t)\sim (x^a(t),y^i(t))$ is an admissible curve and $\xi(t)\sim(x^a(t),f^i(t))$ a generator of the infinitesimal variation $\del_\xi\gamma$ (see Definition \ref{def:variation}) then
$$\del_\zeta\gamma\sim\left(x^a(t),y^i(t),\rho^a_i(x(t))f^i(t), \dot f^j(t)+c^j_{kl}(x(t))y^k(t)f^l(t)\right),$$
where $\gamma(t)$ satisfies additionally the admissibility condition 
$$\dot x^a(t)=\rho^a_i(x(t))y^i(t).$$

\paragraph{Canonical DVB isomorphisms.} For a particular example of the tangent bundle Lie algebroid $\tau_Q:\T Q\ra Q$ the coefficients of the anchor and the bracket in canonical coordinates $(q^\alpha,\dot q^\beta)$ are simply $\rho^\alpha_\beta(q)=\del^\alpha_\beta$ and $c^\alpha_{\beta\gamma}(q)=0$. Hence we have the following coordinate formulas for canonical DVB morphism $\alpha_Q:\T\T^\ast Q\ra\T^\ast\T Q$ and $\kappa_Q:\T\T Q\ra\T\T Q$:
\begin{align*}
&\alpha_Q:\left(q^\alpha,p_\beta,\dot q^\gamma,\dot p_\delta\right)\longmapsto\left(q^\alpha,\dot q^\beta,\dot p_\gamma,p_\delta\right),\\
&\kappa_Q:\left(q^\alpha,\dot q^\beta,\del q^\gamma,\del\dot q^\delta\right)\longmapsto\left(q^\alpha,\del q^\beta,\dot q^\gamma,\del\dot q^\delta\right).
\end{align*}

\paragraph{Tangent lifts.}
For a function $F:Q\ra\R$ its tangent lift $\dT F:\T Q\ra \R$ reads locally as
$$\dT F(q,\dot q)=\frac{\pa F}{\pa q^\alpha}(q)\dot q^\alpha$$
and $\dT^2 F:\T\T Q\ra\R$ as
$$\dT^2 F(q,\dot q,\del q,\del\dot q)=\frac{\pa F}{\pa q^\alpha}\del\dot q^\alpha+\frac{\pa^2 F}{\pa q^\beta\pa q^\alpha}(q)\del q^\beta\dot q^\alpha.$$
The $\kappa_Q$-invariance of $\dT^2 F$ follows from the symmetry of the second derivative $\frac{\del^2 F}{\pa q^\alpha\pa q^\beta}(q)=\frac{\del^2 F}{\pa q^\beta\pa q^\alpha}(q)$. 

The tangent lift $\dT \Pi\in\X^2(\T E^\ast)$ (see \eqref{eqn:lift_bivector}) of the bivector $\Pi\in\X^2(E^\ast)$ given by \eqref{eqn:Pi_coor} reads as (see \eqref{eqn:lift_coor}) 
\begin{equation}\label{eqn:Lift_Pi_coor}
\begin{split}
\dT \Pi(x,z,\dot x,\dot z)=\frac 12\left( c^k_{ij}(x)\dot z_k+\frac{\pa c^k_{ij}}{\pa x^a}(x)z_k\dot x^a\right)\pa_{\dot z_i}\wedge\pa_{\dot z_j}+c^k_{ij}(x)z_k\pa_{z_i}\wedge\pa_{\dot z_j}+\\
\frac{\pa\rho^b_i}{\pa x^a}(x)\dot x^a\pa_{\dot x^b}\wedge\pa_{\dot z_i}+\rho^b_i(x)\pa_{x^b}\wedge\pa_{\dot z_i}+\rho^b_i(x)\pa_{\dot x^b}\wedge\pa_{z_i}.
\end{split}
\end{equation}

\paragraph{EL and Jacobi equations.}
For a Lagrangian function $L:E\ra\R$ the maps $\lambda_L$ and $\Lambda_L$ read as
\begin{align*}
&\lambda_L(x,y)\sim\left(x^a,\frac{\pa L}{\pa y^i}(x,y)\right),\\
&\Lambda_L(x,y)\sim\left(x^a,\frac{\pa L}{\pa y^i}(x,y),\rho^b_i(x)y^k,c^k_{ij}(x)y^i\frac{\pa L}{\pa y^k}(x,y)+\rho^a_j(x)\frac{\pa L}{\pa x^a}(x,y)\right).
\end{align*}
Now the EL equation \eqref{eqn:EL} for $\gamma(t)\sim(x^a(t),y^i(t))$ takes the following form
\begin{align}
\label{eqn:ELadm_coor} &\frac{\dd}{\dd t} x^a(t)=\rho^a_i(x(t))y^i(t)\\
&\frac{\dd}{\dd t}\left(\frac{\pa L}{\pa y^i}(x(t),y(t))\right)=\rho^a_i(x(t))\frac{\pa L}{\pa x^a}(x(t),y(t))+c^k_{ji}(x(t))y^j(t)\frac{\pa L}{\pa 
y^k}(x(t),y(t)).\label{eqn:EL_coor}\end{align} 
We see that \eqref{eqn:ELadm_coor} implies admissibility of $\gamma(t)$. 

The Jacobi equation \eqref{eqn:Jacobi_2} for a JVF $\eta(t)\sim (x^a(t),\xi^i(t))$ along $\gamma(t)\sim (x^a(t),y^i(t))$ takes the following form
\begin{equation}\label{eqn:Jacobi1_coor}
\rho^a_k(x)c^k_{ij}(x)y^j\xi^i+\frac{\pa\rho^a_i}{\pa x^b}(x)\rho^b_j(x)y^j\xi^i=\frac{\pa \rho^a_j}{\pa x^b}\rho^b_i(x)y^j\xi^i
\end{equation}
and
\begin{equation}\label{eqn:Jacobi2_coor}
\begin{split}
\frac{\dd }{\dd t}\left(\frac{\pa^2 L}{\pa x^a y^j}\rho^a_i(x)\xi^i+\frac{\pa^2L}{\pa y^s\pa y^j}\left(\dot\xi^s+c^s_{ik}(x)y^i\xi^k\right)\right)=\\
\rho^a_i(x)\xi^i\left(\frac{\pa c^k_{sj}}{\pa x^a}(x)y^s\frac{\pa L}{\pa y^k}+c^k_{sj}(x)y^s\frac{\pa^2 L}{\pa x^a\pa y^k}+\frac{\pa \rho^b_j}{\pa x^a}(x)\frac{\pa L}{\pa x^b}+\rho^b_j(x)\frac{\pa^2 L}{\pa x^a\pa x^b}\right)+\\
\left(\dot \xi^s+c^s_{ik}(x)y^i\xi^k\right)\left(c^k_{sj}(x)\frac{\pa L}{\pa y^k}+c^k_{ij}(x)y^i\frac{\pa^2 L}{\pa y^s\pa^k}+\rho^a_j(x)\frac{\pa^2 L}{\pa y^s\pa x^a}\right),
\end{split}
\end{equation}
where $\gamma(t)$ satisfies \eqref{eqn:ELadm_coor} and \eqref{eqn:EL_coor}. observe that \eqref{eqn:Jacobi1_coor} is just \eqref{eqn:ala_coor} multiplied by $y^j$ and $\xi^i$. This is in agreement with our considerations form Remark \ref{rem:Jacobi_adm}.  

\paragraph{The second variation.} At the end of our considerations we will give the coordinate description of vectors $\del_{(\Delta\xi,\eta)}\gamma$ which generate second variations. We will be using induced coordinates
$$(x^a,y^i,\dot x^b,\dot y^j
,\del x^c,\del y^k,\del\dot x^d,\del\dot y^l)$$ on $\T\T E$. Our ingredients are:  \begin{itemize}
\item an admissible curve $\gamma(t)\sim(x^a(t),y^i(t))\in E$ i.e.
$\dot x^a=\rho^a_i(x)y^i$,
\item curves $\xi(t)\sim(x^a(t),f^i(t))$ and $\eta(t)\sim(x^a(t),h^i(t))$ from $E_{\tau\circ\gamma}$ and
\item a curve $\Delta\xi(t)\sim(x^a(t),f^i(t),\Delta x^b(t),\Delta f^j(t))\in\T_{\xi}E$ such that $\T\tau\circ\Delta\xi=\rho(\eta)$ i.e.
$\Delta x^a=\rho^a_i(x)h^i$.
\end{itemize}
Now $\del_{(\Delta\xi,\eta)}\gamma$ equals
\begin{equation}\label{eqn:del2_coor}
\left(x^a,y^i,\rho^b_i(x)f^i,\dot f^j+c^j_{kl}(x)y^kj^l, \rho^c_i(x)h^i,\dot h^k+c^k_{ls}(x)y^lh^s,\frac{\pa \rho^d_i}{\pa x^a}\rho^a_j(x)h^jf^i+\rho^d_i(x)\Delta f^i,\Delta y^l \right),
\end{equation} 
where $$\Delta y^l:=\dot{\Delta f^l}+c^l_{ij}(x)y^i\Delta f^j+c^l_{ij}(x)(\dot h^i+c^i_{st}(x)y^sh^t)f^j+\frac{\pa c^l_{ij}}{\pa x^a}(x)\rho^a_k(x)h^ky^if^j.$$

\section{Examples}\label{sec:examples}

\subsection{Riemannian geometry}\label{ssec:riem_geom}

\paragraph{$E$-connections.} Consider an AL algebroid structure on $\tau:E\ra M$ with the anchor $\rho$ and the bracket $[\cdot,\cdot]$. An \emph{$E$-connection} is a bilinear map 
$$\nabla:\Sec(E)\times\Sec(E)\lra\Sec(E)$$
satisfying the following properties:
\begin{align*}
&\nabla_{fX}Y=f\nabla_XY \\
&\nabla_XfY=f\nabla_XY+\rho(x)(f)Y,
\end{align*}
for every $X,Y\in\Sec(E)$ and $f\in C^\infty(M)$. Observe that even though  
$\nabla$ is defined on sections of $E$, to compute $\nabla_XY$at $p\in M$ it is enough to know $X(p)$ and the derivative of $Y$ in the direction $\rho(X)(p)$. 

An $E$-connection is called \emph{torsion-free} (compatible with the algebroid structure) if
$$\nabla_XY-\nabla_YX=[X,Y].$$
If $E$ is equipped with a Riemannian metric, i.e. a fiber-wise linear symmetric and non-degenerate map $g:E\times_M E\ra \R$, we can speak about \emph{metric} $E$-connections which satisfy
$$\rho(X)g(Y,Z)=g(\nabla_XY,Z)+g(Y,\nabla_XZ).$$

The following well-known result generalizes the fundamental fact form Riemannian geometry.

\begin{lem}\label{lem:nabla_LC}
Let $(E,g)$ be as above.  Then the formula:
\begin{align*}
2g(\nabla_XY,Z):=&\rho(X)g(Y,Z)+\rho(Y)g(X,Z)-\rho(Z)g(X,Y)+\\
&g([X,Y],Z)-g([Y,Z],X)+g([Z,X],Y)
\end{align*}
defines a unique metric and torsion-free $E$-connection.
\end{lem}

Given an $E$-connection $\nabla$ we can define its \emph{curvature}: 
$$R_\nabla(X,Y)Z:=\nabla_X\nabla_YZ-\nabla_Y\nabla_XZ-\nabla_{[X,Y]}Z.$$
A straightforward computation shows that $R_\nabla(X,Y)Z$ is tensorial in the first two entries and tensorial w.r.t. $Z$ whenever $E$ is an AL algebroid. 

\paragraph{A geodesic problem.} Given $E$ and $g$ as above we can consider a natural Lagrangian system $L:E\ra\R$ on $E$ given by the energy function  $L(a):=\frac 12g(a,a)$. Our goal is to describe the EL and the Jacobi equation for this system. We will show that the EL equation for an admissible curve $\gamma\in E$ reads as
$$\nabla_\gamma\gamma=0$$
and the Jacobi equation for $\eta\in E_{\tau\circ\gamma}$ as
$$\nabla_\gamma\nabla_\gamma\eta+R(\eta,\gamma)\gamma=0.$$
Of course in the above equations $\nabla$ is the canonical connection described in Lemma \ref{lem:nabla_LC} and $R$ its curvature. obviously these equations are straightforward generalizations of classical formulas from Riemannian geometry.

\paragraph{Calculations.}
In our considerations we will be working with geometric objects defined along an admissible curve $\gamma(t)\in E$. However sometimes we will be using objects defined locally around $\gamma(t)$, such as the algebroid bracket $[\cdot,\cdot]$ and the covariant derivative $\nabla$. Therefore we will need the following construction: if $\xi(t)$ is a curve in $E$ over $x(t)\in M$ and $\Delta \xi\in\T_\xi E$ a vector field along $\xi$, then by a \emph{$\Delta\xi$-extension of $\xi$} we will understand any local section $\ol\xi:M\ra E$ around $\xi$ such that $\ol\xi(x(t))=\xi(t)$ and $\T\ol\xi(\T\tau(\Delta\xi(t)))=\Delta\xi(t)$. 

In particular  let $\gamma(t)\in E$ be an admissible curve over $x(t)\in M$ and $\eta\in E_{\tau\circ\gamma}$ a generator of an admissible variation $\del_\eta\gamma$. Then for $\ol\gamma$ -- a $\del_\eta\gamma$-extension of $\gamma$ and $\ol\eta$ -- any extension of $\eta$ we got
\begin{equation}\label{eqn:bracket_0}
[\ol\eta,\ol\gamma]=0\quad\text{along $x(t)$.}
\end{equation}
This fact is just a matter of a simple coordinate calculation.

Now we are ready to derive the EL equation. We claim that for an admissible curve $\gamma(t)\in E$ we have
\begin{align}
&\lambda_L(\gamma(t))=\wt g(\gamma(t))\label{eqn:lambda_geod}
\intertext{and}
&\Lambda_L(\gamma(t))=\T\wt g\left(\dot\gamma(t)-V(\gamma(t),\nabla_\gamma\gamma(t))\right),\label{eqn:Lambda_geod}
\end{align}
where $\wt g:E\ni a\mapsto g(a,\cdot)\in E^\ast$ is a natural isomorphism and $V(a,b)\in\V_a E\subset\T_a E$ is a vertical vector corresponding to $(a,b)$ in the canonical identification $\V E\approx E\times_ME$. Note that \eqref{eqn:Lambda_geod} is tensorial w.r.t. $\gamma$ although the RHS contains $\dot\gamma$.

Formula \eqref{eqn:lambda_geod} can be checked directly in coordinates. To prove \eqref{eqn:Lambda_geod} we will do the following reasoning. For any admissible $\gamma(t)\in E$ over $x(t)\in M$ and any generator $\xi(t)\in E_{\tau\circ\gamma(t)}$:
\begin{align*}
\<\dd S_L(\gamma),\del_\xi\gamma>=\int_{t_0}^{t_1}\<\dd L(\gamma(t)),\del_\xi\gamma(t)>\dd t=\int_{t_0}^{t_1}\frac 12\T g(\del_\xi\gamma(t),\del_\xi\gamma(t))\dd t=
\int_{t_0}^{t_1}\frac 12\rho(\xi)g(\ol\gamma,\ol\gamma)(x(t))\dd t,
\end{align*}
where $\ol\gamma$ is an $\del_\xi\gamma$-extension of $\gamma$. The later equals
\begin{align*}
\int_{t_0}^{t_1}g(\nabla_\xi\ol\gamma,\ol\gamma)(x(t))\dd t=\int_{t_0}^{t_1}g([\ol\xi,\ol\gamma]+\nabla_\gamma\ol\xi,\ol\gamma)(x(t))\dd t\overset{\eqref{eqn:bracket_0}}=\int_{t_0}^{t_1}g(\nabla_\gamma\ol\xi,\ol\gamma)(x(t))\dd t=\\
\int_{t_0}^{t_1}\rho(\gamma)g(\xi,\gamma)(x(t))\dd t-\int_{t_0}^{t_1}g(\xi,\nabla_\gamma\gamma)(x(t))\dd t.
\end{align*}
Since $\gamma$ is admissible $\rho(\gamma)g(\xi,\gamma)(x(t))=\frac{\dd}{\dd t}g(\xi(t),\gamma(t))$; and hence
\begin{align*}
\<\dd S_(\gamma),\del_\xi\gamma>=&g(\xi(t),\gamma(t))\bigg|_{t_0}^{t_1}-\int_{t_0}^{t_1}g(\xi,\nabla_\gamma\gamma)(x(t))\dd t=\\
&\<\xi(t),\lambda_L(\gamma(y))>\bigg|_{t_0}^{t_1}-\int_{t_0}^{t_1}g(\xi,\nabla_\gamma\gamma)(x(t))\dd t.
\end{align*}
Comparing the above with \eqref{eqn:EL_var} we get
$$\<\xi(t),\VV_{\Lambda_L(\gamma(t))-\frac{\dd}{\dd t}\lambda_L(\gamma(t))}>=-g(\xi(t),\nabla_\gamma\gamma(t)).$$
Obviously 
\begin{align*}
\<\xi(t),\VV_{\Lambda_L(\gamma(t))-\frac{\dd}{\dd t}\lambda_L(\gamma(T))}>=&-g(\xi(t),\nabla_\gamma\gamma(t))=\<\xi(t),\wt g(-\nabla_\gamma\gamma(t))>=\\ &\<\xi(t),\VV_{\T\wt  g\left(\dot\gamma(t)-V(\gamma(t),\nabla_\gamma\gamma(t))-\dot\gamma(t)\right)}>=
\<\xi(t),\VV_{\T\wt g\left(\dot\gamma(t)-V(\gamma(t),\nabla_\gamma\gamma(t))\right)-\frac{\dd}{\dd t}\lambda_L(\gamma(t))}>.
\end{align*}
It follows that \eqref{eqn:Lambda_geod} holds. In particular we know that the EL equation reads as
$$\T\wt g\left(V(\gamma(t),\nabla_\gamma\gamma(t))\right)=0,$$
which is equivalent to $\nabla_\gamma\gamma(t)=0$.

Now we are ready to derive the Jacobi equation. According to \eqref{eqn:Jacobi_2} we need only to compute the tangent maps of $\lambda_L$ and $\Lambda_L$ at $\del_\eta\gamma$. Obviously
$$\T\lambda_L(\Delta\gamma)=\T\wt g(\Delta\gamma).$$
Now
\begin{align*}
\kappa_{E^\ast}\circ\T\Lambda_L(\Delta\gamma)=&\kappa_{E^\ast}\circ\T\T\wt g\left(\Delta(\dot\gamma)-\Delta V(\gamma,\Delta_\gamma\gamma)\right)=\\
&\T\T\wt g\circ\kappa_E\left(\Delta(\dot\gamma)-\Delta V(\gamma,\Delta_\gamma\gamma)\right)=
\T\T\wt g\left((\Delta\gamma)^{\cdot}-V(\Delta\gamma,\Delta(\nabla_\gamma\gamma))\right).
\end{align*}
We deduce that the Jacobi equation for $\eta\in E_{\tau\gamma}$ along a geodesics $\gamma$  takes the form $\T\T\wt g\left(V(\Delta\gamma,\Delta(\nabla_\gamma\gamma))\right)=0$ at $\Delta\gamma=\del_\eta\gamma$, which is equivalent to
$$\Delta(\nabla_\gamma\gamma)=0,$$
where $\Delta\gamma=\del_\eta\gamma$ and $\nabla_\gamma\gamma=0$. 
Now if $\Delta\gamma=\del_\eta\gamma$ then
$$\<\Delta\xi,\VV_{\T\T\wt g(V(\Delta\gamma,\Delta(\nabla_\gamma\gamma)))}>=\T g(\Delta\xi,\Delta(\nabla_\gamma\gamma))=\rho(\eta)g(\ol\xi,\nabla_{\ol\gamma}\ol\gamma),$$
where $\ol\gamma$ is a $\del_\eta\gamma$-extension of $\gamma$ and $\ol\xi$ is a $\Delta\xi$-extension of $\xi$. The above equals
$$g(\nabla_\eta\ol\xi,\nabla_\gamma\gamma)+g(\eta,\nabla_\eta\nabla_{\ol\gamma}\ol\gamma)=g(\eta,\nabla_\eta\nabla_{\ol\gamma}\ol\gamma),$$
since $\nabla_\gamma\gamma=0$. Using the definition of the curvature we get
\begin{align*}
g\left(\xi,\nabla_\eta\nabla_{\ol\gamma}{\ol\gamma}\right)=g\left(\xi,R(\eta,\gamma)\gamma+\nabla_{\ol\gamma}\nabla_{\ol\eta}{\ol\gamma}+\nabla_{[\ol\eta,\ol\gamma]}\ol\gamma\right)\overset{\eqref{eqn:bracket_0}}=g\left(\xi,R(\eta,\gamma)\gamma+\nabla_{\ol\gamma}\nabla_{\ol\eta}{\ol\gamma}\right)=\\
g\left(\xi,R(\eta,\gamma)\gamma+\nabla_{\ol\gamma}\nabla_{\ol\gamma}{\ol\eta}+\nabla_{\ol\gamma}[\ol\eta,\ol\gamma]\right)\overset{\eqref{eqn:bracket_0}}=g\left(\xi,R(\eta,\gamma)\gamma+\nabla_{\gamma}\nabla_{\gamma}{\eta}\right)
\end{align*}
In this way we have proved that 
$$R(\nabla,\gamma)\gamma+\nabla_\gamma\nabla_\gamma\eta=0$$
is equivalent to the Jacobi equation and, simultaneously, we have proved the generalization of \eqref{eqn:C}.
 
\subsection{Euler-Poincare equations}\label{ssec:EP_eqn}
\paragraph{Tangent lift $\dT\g$.}
Let $(\g,[\cdot,\cdot])$ be a skew-symmetric algebra. The vector space $\g$ treated as a vector bundle over a point, together with the bracket $[\cdot,\cdot]$ and the trivial anchor is an AL algebroid (condition \eqref{eqn:ala} is trivially satisfied). It is a Lie algebroid if  $(\g,[\cdot,\cdot])$ is a Lie algebra. We shall now describe the lifted algebroid structure $\dT\g$. It is also given by a skew-symmetric algebra structure on $\T\g$, as the tangent lift of the vector bundle $\g\ra\{\ast\}$ is again a VB over a point. Using the canonical identification $\T\g\approx\g\times\g$ we obtain a skew-symmetric algebra structure  $(\g\times\g,[\cdot,\cdot]_\dT)$.

The canonical isomorphism $\eps:\T^\ast\g\approx\g\times\g^\ast\lra\g^\ast\times\g^\ast\approx\T\g^\ast$ for the algebroid  $(\g,[\cdot,\cdot])$ reads as
\begin{equation}\label{eqn:eps_g}
\eps:(a,z)\longmapsto(z,\<z,[a,\cdot]>);
\end{equation}
where $\<\cdot,\cdot>:\g^\ast\times\g\ra\R$ is the canonical parring. It follows that for an admissible curve $\gamma=a(t)\in\g$ (since the anchor is trivial every path is $\g$-admissible) and a generator $\xi=f(t)\in\g$ the associated admissible variation reads as
\begin{equation}\label{eqn:adm_var_g}
\del_\xi\gamma=(a(t),\dot f(t)+[a(t),f(t)])\in\g\times\g\approx\T\g.
\end{equation}
Proceeding analogously as in Section \ref{sec:variations} we can differentiate a 1-parameter family of such variations to get
$$\Delta(\del_\xi\gamma)=\left(a,\dot f+[a,f],\Delta a,\dot{\Delta f}+[\Delta a,f]+[a,\Delta f]\right)(t)\in\g\times\g\times\g\times\g\approx\T\T\g.$$
On the other hand we know form Section \ref{sec:variations} that $\Delta(\del_\xi\gamma)$ is just $\kappa_\g\circ\del_{\Delta\xi}\Delta\gamma$, where $\del_{\Delta\xi}\Delta\gamma$ is a $\dT\g$-admissible variation along $\Delta\gamma=(a(t),\Delta a(t))$ generated by $\Delta\xi=(f(t),\Delta f(t))$. It turns out that 
\begin{equation}\label{eqn:adm_var_dTg}\del_{\Delta \xi}\Delta\gamma=\left(a,\Delta a,\dot f+[a,f],\dot{\Delta f}+[\Delta a,f]+[a,\Delta f]\right)(t)\in\g\times\g\times\g\times\g\approx\T\T\g.
\end{equation}
Formula \eqref{eqn:adm_var_g} for $(\g\times\g,[\cdot,\cdot]_\dT)$ describes the relation between the skew-algebra bracket and admissible variations. 
Hence from \eqref{eqn:adm_var_dTg} we deduce that 
\begin{equation}\label{eqn:bracket_g}[(a,\Delta a),(f,\Delta f)]_\dT=\left([a,f],[\Delta a,f]+[a,\Delta f]\right).
\end{equation}

\paragraph{The EL and Jacobi equations.}
Let us now describe the EL and Jacobi equations for a Lagrangian system $L:\g\ra\R$ on $\g$. Observe that $\dd L:\g\ra\T^\ast\g\approx\g\times\g^\ast$ defines a map $\del L:\g\ra\g^\ast$ given by $\dd L(a)=(a,\del L(a))$ (in fact $\del L=\lambda_L$ in the terminology of Section \ref{sec:EL}). From \eqref{eqn:eps_g} we deduce that the EL equation reads as
\begin{equation}\label{eqn:EL_g}
\frac{\dd}{\dd t}\left(\del L(a(t))\right)=\<\del L(a(t)),[a(t),\cdot]>.
\end{equation}
After introducing linear coordinates $(y^i)\ i=1,\hdots,k$ on $\g$ (denote by $c^i_{jk}$ the coefficients of the bracket in these coordinates) this equation takes the form
$$\frac{\dd}{\dd t}\left(\frac{\pa L}{\pa y^j}(y(t))\right)=c^k_{ij}\frac{\pa L}{\pa y^k}(y(t))y^i(t).$$
We recognize the Euler-Poincare equation. 

Since $\dT L(a,b)=\<\del L(a),b>$ then 
$$\del\dT L(a,b)=\left(\del L(a),\del^2L(a,b)\right)\in\g^\ast\times\g^\ast\approx\T\g^\ast,$$
where $\del^2 L:\g\times\g\ra\g^\ast$ is a map linear in the second entry. 
Formula \eqref{eqn:EL_g} together with expression \eqref{eqn:bracket_g} will give us the EL equation for the lifted system $\dT L:\T\g\ra\R$ on $\dT\g$. Then form Theorem \ref{thm:Jacobi} we can deduce the Jacobi equation for a Jacobi vector field $(a(t),h(t))\in\g\times\g\approx\T\g$ along $a(t)\in\g$:
\begin{equation}\label{eqn:Jacobi_g}
\begin{split}
&\frac{\dd}{\dd t}\left(\del L(a)\right)=\<\del L(a),[a,\cdot]>\\
&\frac{\dd}{\dd t}\left(\del^2 L(a,\dot h+[a,h])\right)=\<\del L(a),[\dot h+[a,h],\cdot]>+\<\del^2 L(a,\dot h+[a,h]),[a,\cdot]>.
\end{split}
\end{equation}
The first equation is just the EL equation for the Lagrangian system $L:\g\ra\R$. The second equation in coordinates $a(t)\sim(y^i(t))$ and $h(t)\sim (h^i(t))$ reads as
$$\frac{\dd}{\dd t}\left(\frac{\pa^2 L}{\pa y^j\pa y^s}(y)J^s\right)=\left(c^k_{sj}\frac{\pa L}{\pa y^k}(y)+c^k_{ij}y^i\frac{\pa^2 L}{\pa y^s\pa y^k}(y)\right)J^s,$$
where $J^s=\dot h^s+c^s_{kl}a^kh^l$. In a particular case $L(a)=\frac 12 g(a,a)$, where $g:\g\times\g\ra\R$ is a symmetric 2-tensor, we get $\del L(a)=g(a,\cdot)$ and $\del^2L(a,b)=g(b,\cdot)$. In this situation equation \eqref{eqn:Jacobi_g} takes the form
\begin{align*}
&g(\dot a,\cdot)=g(a,[a,\cdot])\\
&g(\ddot h+[a,\dot h]+[\dot a,h],\cdot)=g(a,[\dot h+[a,h],\cdot])+g(\dot h+[a,h],[a,\cdot]).
\end{align*}

\paragraph{Symmetry of the second variation.} Finally let us study the symmetry of the second variation of the action 
$$\g\ni\gamma(t)\longmapsto S_L(\gamma)=\int_{t_0}^{t_1}L(\gamma(t))\dd t.$$
We know that for $\eta$ and $\xi$ vanishing at the end-points 
$$\del^2 L(\gamma)(\eta,\xi)=\int_{t_0}^{t_1}\dT^2 L(\del_{(\Delta\xi,\eta)})\dd t,$$
where $\Delta\xi$ vanishes at the end-points. Since $\del_{(\Delta\xi,\eta)}\gamma=\kappa_\g\circ\del_{\Delta\xi}(\del_\eta\gamma)$ from \eqref{eqn:adm_var_g} and \eqref{eqn:adm_var_dTg} we get
$$\del_{(\Delta\xi,\eta)}\gamma=\left(a,\dot f+[a,f],\dot h+[a,h],\dot{\Delta f}+[\dot h+[a,h],f]+[a,\Delta f]\right)(t),$$
where $\gamma=a(t)$, $\eta=h(t)$, $\xi=f(t)$ and $\Delta\xi=(f(t),\Delta f(t))$. The curve $\Delta\eta=(h(t),\Delta f(t)+[h(t),f(t)])$ is $\kappa$-related to $\Delta\xi$ and also vanishes at the end-points, since $\Delta\xi$, $\xi$ and $\eta$ do. We have
$$\del_{(\Delta\eta,\xi)}\gamma=\left(a,\dot h+[a,h],\dot f+[a,f],\dot{\Delta f}+[\dot h,f]+[h,\dot f]+[\dot f+[a,f],h]+[a,\Delta f+[h,f]]\right)(t).$$
A short calculation shows that 
$$\del_{(\Delta\xi,\eta)}\gamma-\kappa_\g\circ\del_{(\Delta\eta,\xi)}\gamma=\left(a,\dot f+[a,f],0,J(a,h,f)\right)(t),$$
where $J(a,h,f):=[[a,h],f]-[a,[h,f]]+[h,[a,f]]$ is the Jacobiator. Now
\begin{align*}
&\del^2 S_L(\gamma)(\eta,\xi)-\del^2 S_L(\gamma)(\xi,\eta)=\int_{t_0}^{t_1}\dT^2 L(\del_{(\Delta\xi,\eta)}\gamma)\dd t-\int_{t_0}^{t_1}\dT^2 L(\del_{(\Delta\eta,\xi)}\gamma)\dd t=\\
&\int_{t_0}^{t_1}\dT^2 L(\del_{(\Delta\xi,\eta)}\gamma)\dd t-\int_{t_0}^{t_1}\dT^2 L(\kappa_\g\circ\del_{(\Delta\eta,\xi)}\gamma)\dd t=\int_{t_0}^{t_1}\<\dd\dT L(\del_\xi\gamma), \del_{(\Delta\xi,\eta)}\gamma-\kappa_\g\circ\del_{(\Delta\eta,\xi)}\gamma>\dd t=\\
&\int_{t_0}^{t_1}\<d L(\gamma), J(\gamma,\eta,\xi)>\dd t.
\end{align*}
We see that if $(\g,[\cdot,\cdot])$ is not a Lie algebra (i.e. the bracket $[\cdot,\cdot]$ does not satisfy the Jacobi identity) then for a generic Lagrangian $\del^2 S_L(\gamma)$ will not be symmetric. 

\section*{Final remarks}
\paragraph{Further study, Hamiltonian dynamics.}
Let us note that the classical geodesic equation can be regarded not only as a EL equation on $\T M$, but also as a Hamilton equation on $\T^\ast M$. In fact most of the results of this paper  can be repeated for Hamiltonian systems on algebroids. The results of Theorems \ref{thm:main} and \ref{thm:main1} hold in the Hamiltonian case after substituting ''EL equation'' by ''Hamilton equation'' and changing the standard variational principle to a more suitable one. A detailed study of this case will be given in a forthcoming publication \cite{JA_Jacobi_ham}. 
  
\paragraph{Curvature?}
It would be interesting to extend other aspects of the classical theory of Jacobi fields and conjugate points to the geometric setting presented in this paper. In particular it is tempting to find an analog of curvature for arbitrary Lagrangian (or Hamiltonian systems) and use it to prove results about existence or non-existence of conjugate points (see e.g. \cite{HE} for such results with applications to general relativity). The results of this kind (especially in the Hamiltonian context) could be useful in the study of Hamilton-Jacobi-Bellman equation in optimal control theory (see e.g. \cite{Fr}). It seems to us that similar ideas could be used to study phase transitions in thermodynamics.

So far we were unable to find any satisfactory definition of curvature. It is highly probable that a general definition of this kind does not exists and some additional assumptions about the nature of the systems in consideration should be made. For example the yet unpublished paper \cite{JK} gives a definition of the curvature for Hamiltonian systems if the system satisfies certain geometric conditions. In the simplest case of a Hamiltonian system on a cotangent bundle $\T^\ast M$ this condition is the hyper-regularity of the Hamiltonian.



\end{document}